\documentclass[12pt]{amsart}
\usepackage{amssymb}
\usepackage{graphicx}

\newcommand{\N}{{\mathbb N}}
\newcommand{\C}{{\mathbb C}}
\newcommand{\Z}{{\mathbb Z}}

\newcommand{\D}{{\mathbb D}}

\newcommand{\wt}{\widetilde}

\newcommand{\tef}{transcendental entire function}

\newcommand{\mconn}{multiply connected}

\newcommand{\mcfc}{multiply connected wandering domain}
\newcommand{\sw}{spider's web}

\newcommand\eps{\varepsilon}
\newcommand\qfor{\;\;\text{for }}
\newcommand{\diam}{{\rm diam}}
\newcommand{\ol}{\overline}
\newcommand{\ucc}{unbounded closed connected}
\newcommand{\ucm}{unbounded continuum}
\newcommand{\uca}{unbounded continua}
\newcommand{\FC}{\Phi}

\theoremstyle{plain}
\newtheorem{theorem}{Theorem}[section]

\newtheorem{lemma}[theorem]{Lemma}     % 2nd argument is what is printed
\newtheorem{corollary}[theorem]{Corollary}\theoremstyle{definition}

\theoremstyle{remark}

\theoremstyle{problem}

\theoremstyle{question}
\newtheorem{question}{Question}[section]
\theoremstyle{example}

\begin{document}

\title[Eremenko points and the structure of the escaping set]{Eremenko points and the structure of the escaping set}

\author{P. J. Rippon}
\address{Department of Mathematics and Statistics\\
The Open University \\
Walton Hall\\
Milton Keynes MK7 6AA\\
UK}
\email{p.j.rippon@open.ac.uk}

\author{G. M. Stallard}
\address{Department of Mathematics and Statistics\\
The Open University \\
Walton Hall\\
Milton Keynes MK7 6AA\\
UK}
\email{g.m.stallard@open.ac.uk}

%\thanks{The second author <... thanks>}

%End Two authors

%\keywords{asymptotic tracts}
\subjclass{30D05, 37F10.\newline\hspace*{.32cm} Both authors were
supported by EPSRC grant EP/K031163/1.}
%\date{1999}

%End topmatter

\begin{abstract}
Much recent work on the iterates of a transcendental entire function~$f$ has been motivated by Eremenko's conjecture that all the components of the escaping set $I(f)$ are unbounded. Here we show that if $I(f)$ is disconnected, then the set $I(f)\setminus D$ has uncountably many unbounded components for any open disc~$D$ that meets the Julia set of~$f$. For the set $A_R(f)$, which is the `core' of the fast escaping set, we prove the much stronger result that for some $R>0$ either $A_R(f)$ is connected and has the structure of an infinite spider's web or it has uncountably many components each of which is unbounded. There are analogous results for the intersections of these sets with the Julia set when no multiply connected wandering domains are present, but strikingly different results when they are present. In proving these, we obtain the unexpected result that multiply connected wandering domains can have complementary components with no interior, indeed uncountably many.
\end{abstract}

\maketitle

\section{Introduction}
\setcounter{equation}{0}
Let~$f$ be a {\tef} and denote by $f^{n},\,n=0,1,2,\ldots\,$, the $n$th iterate of~$f$. The {\it Fatou set} $F(f)$ is defined to be the set of points $z \in \C$ such that $(f^{n})_{n \in \N}$ forms a normal family in some neighborhood of~$z$, and the {\it Julia set} of~$f$ is the complement of $F(f)$. The components of $F(f)$ are called {\it Fatou components}. An introduction to the properties of these sets can be found in~\cite{wB93}.

The {\it escaping set}
\[
 I(f) = \{z: f^n(z) \to \infty \mbox{ as } n \to \infty \}
\]
was first studied in detail by Eremenko~\cite{aE89}, who made what is known as `Eremenko's conjecture', namely that all the components of $I(f)$ are unbounded. This conjecture remains unsolved in spite of much attention and many partial results; see \cite{RRRS}, \cite{lR07} and \cite{RS09}, for example.

The strongest general results about Eremenko's conjecture have been obtained by using the subset of $I(f)$ called the {\it fast escaping set} $A(f)$, introduced in \cite{BH99}, which can be defined as follows; see \cite{RS09}. First put
\begin{equation}\label{ARfdef}
A_R(f) = \{z:|f^n(z)| \geq M^n(R), \text{ for } n \in \N\},
\end{equation}
where $M(r)=M(r,f)=\max \{|f(z)|:|z|=r\},\;r>0$, $M^n(r)=M^n(r,f)$ denotes the $n$-th iterate of $r\mapsto M(r,f)$, and $R>0$ is so large that $M(r)>r$ for $r\ge R$, and then put
\[
A(f)=\{z: {\rm for\;some}\; \ell\in \N, f^{\ell}(z)\in A_R(f)\}.
\]
This definition of $A(f)$ is independent of~$R$.

The set $A(f)$ has stronger properties than $I(f)$ and these have been used to prove significant facts about the structure of $I(f)$, such as the following:
\begin{itemize}
\item
$I(f)$ has at least one unbounded component -- this follows immediately from the fact that all the components of $A(f)$ (and indeed of $A_R(f)$) are unbounded \cite[Theorem~1.1]{RS09};
\item
more precisely, either $I(f)$ is connected or it has infinitely many unbounded components \cite[Theorem~1.3]{RS17};
\item
$I(f)\cup\{\infty\}$ is connected \cite[Theorem~4.1]{RS11}.
\end{itemize}
In general, the topological structure of $I(f)$ may be highly complicated; see \cite{lR09}, \cite{RRRS} and \cite{lR11}, for example. However, there are many classes of {\tef}s for which $I(f)$ is connected, in which case Eremenko's conjecture holds in a particularly strong way; see \cite{lR10}, \cite{lR11}, \cite{xJ11} and  \cite{RS09}, for example. In \cite{RS09} we defined a set~$S$ to be an {\it (infinite) spider's web} if~$S$ is connected and there exists a sequence of bounded simply connected domains $G_n, n\in \N$, such that
\[
G_n\subset G_{n+1}\;\text{and}\; \partial G_n \subset S,\;\text{for } n\in\N, \quad \text{and}\quad \bigcup_{n=1}^{\infty}G_n=\C,
\]
and we showed that if the set $A_R(f)$ is a spider's web, for some $R>0$, then so are $A(f)$ and $I(f)$, and moreover $I(f)$ has many strong properties; see also \cite{jO12}.

There are many classes of entire functions for which $A_R(f)$ is a spider's web \cite[Section~8]{RS09}, so $A(f)$ and $I(f)$ are connected for all such functions. On the other hand, there exist entire functions for which $A_R(f)$ is disconnected but $A(f)$ and $I(f)$ are connected, such as $f(z)=\cosh^2 z$ \cite[Section~1]{RS09} and $f(z)=e^z$ \cite{lR11}. There also exist entire functions for which $A_R(f)$ and $A(f)$ are disconnected but $I(f)$ is connected, such as Fatou's function $f(z)=z+1+e^{-z}$ \cite[Exemple~1]{pF26}. Note that for Fatou's function $I(f)$ is a spider's web \cite{vE16}, whereas for $f(z)=e^z$ it is not \cite[Example~2]{OS}.

It is natural to ask how many components the sets $I(f)$, $A(f)$ and $A_R(f)$ can have when these sets are disconnected. For all known examples the answer is `uncountably many' in each case, but in general we know only that for $I(f)$ and $A(f)$ the answer is `infinitely many'; see \cite[Theorem~5.2]{RS11}, where the proof is based on the blowing-up property of $J(f)$. For $J(f)$ itself it is known that either $J(f)$ is connected or it has uncountably many components; see \cite[Theorem~B]{BD00}, where the proof makes strong use of the fact that $J(f)$ is closed and completely invariant.

Even though $I(f)$ can be much more complicated topologically than $J(f)$, it is natural to conjecture that if $I(f)$ is disconnected, then it has uncountably many components, and that the same is true for $A(f)$ and $A_R(f)$. In this paper we make significant progress towards proving these statements, with a particularly striking result for $A_R(f)$. Our first result concerns $I(f)$. Here we use the term {\it {\ucm}} to denote an {\ucc} set.
\begin{theorem}\label{I(f)}
Let~$f$ be a {\tef} and suppose that $I(f)$ is disconnected or, more generally, that there exists an {\ucm} in $I(f)^c$. If~$D$ is any open disc meeting $J(f)$, then the set $I(f)\setminus D$ has uncountably many unbounded components that meet $\partial D$ and these components are separated in $\C\setminus D$ by {\uca} in $I(f)^c$.
\end{theorem}

{\it Remarks}\; 1. Whenever $I(f)$ is disconnected, there exists a closed connected set $\Gamma\subset I(f)^c$ such that $I(f)$ meets at least two complementary components of $\Gamma$; see \cite{lR10}. The set~$\Gamma$ can have no bounded complementary components, by \cite[Theorem~4.1]{RS11}, so it must be an {\ucm}. Thus $I(f)^c$ contains an {\ucm} whenever $I(f)$ is disconnected.

2. Theorem~\ref{I(f)} leaves open the possibility that $I(f)$ could be disconnected and yet have only countably many components, but it shows that in this case the topological structure of $I(f)$ must be extremely complicated.

3. In \cite{ORS16}, we defined a {\it weak spider's web} to be a set with no {\ucm} in its complement. Thus the hypothesis of Theorem~\ref{I(f)} is that $I(f)$ is not a weak spider's web.

There is an analogous result concerning the set $A(f)$, as follows.
\begin{theorem}\label{A(f)}
Let~$f$ be a {\tef} and suppose that $A(f)$ is disconnected or, more generally, that there exists an {\ucm} in $A(f)^c$. If~$D$ is any open disc meeting $J(f)$, then the set $A(f)\setminus D$ has uncountably many unbounded components that meet $\partial D$ and these components are separated in $\C\setminus D$ by {\uca} in $A(f)^c$.
\end{theorem}

For the set $A_R(f)$ we have a much stronger result. To state this, we define
\begin{equation}\label{R(f)}
R(f)=\inf \{R\in [0,\infty): M(r)>r, \text{ for }r \ge R\},
\end{equation}
which is  the least number such that $A_R(f)$ can be defined for all $R>R(f)$.

We recall that if $A_R(f)$ is a {\sw} for some $R>R(f)$, then $A_R(f)$ is a {\sw} for {\it all} $R>R(f)$; see \cite[Lemma~7.1(d)]{RS09}. As noted above there are many classes of entire functions for which $A_R(f)$ is a spider's web, for $R>R(f)$, and there are also many classes for which $A_R(f)$ has uncountably many components. The following theorem shows that for any entire function exactly one of these two extreme situations must occur for many values of~$R$.
\begin{theorem}\label{AR(f)}
Let~$f$ be a {\tef} and let $R(f)$ be given by \eqref{R(f)}. Then either $A_R(f)$ is a {\sw} for all $R>R(f)$ or there is a dense set of values of $R\in (R(f),\infty)$ for which $A_R(f)$ has uncountably many components.
\end{theorem}

We have the following immediate corollary of Theorem~\ref{AR(f)}.
\begin{corollary}\label{uncountable}
Let $f$ be a {\tef}. If $I(f)$ is not a spider's web, then it contains uncountably many disjoint {\uca} all lying in $A_R(f)$.
\end{corollary}

To prove these theorems we start by refining Eremenko's original construction of points in $I(f)$; see \cite{aE89}. We then use this refined construction in different ways to prove Theorem~\ref{I(f)} (and Theorem~\ref{A(f)}), and then Theorem~\ref{AR(f)}. For example, to prove Theorem~\ref{AR(f)} we construct uncountably many points in $I(f)$, each with a distinct itinerary of a certain type. These {\it Eremenko points}, as we call them, all lie in components of $A_R(f)$ for the same value of~$R$, and we use conformal mapping and the theory of prime ends to show that if two of these components corresponding to distinct Eremenko points coincide, then $A_R(f)$ is a spider's web. The fact that $A_R(f)$ is closed is essential to our arguments.

Next, we discuss components of the sets formed by intersecting $I(f)$, $A(f)$ and $A_R(f)$ with the Julia set of~$f$. We recall that if all the Fatou components of~$f$ are simply connected, that is,~$f$ has no {\mcfc}s, then all the components of $J(f)$, and also of $A_R(f)\cap J(f)$ are unbounded; see \cite[Theorem 1.3]{RS09}. Moreover, in this situation all the Eremenko points of~$f$ lie in $J(f)$; see \cite{aE89} and Theorem~\ref{Epoints}~(c) below. Therefore, for the intersections with $J(f)$, our proofs give analogous results to Theorems~\ref{I(f)}, \ref{A(f)} and \ref{AR(f)}, and also Corollary~\ref{uncountable}, such as the following, whose proof we omit.
\begin{theorem}\label{No-mcwd}
Let $f$ be a \tef\ with no {\mcfc}s. Then either $A_R(f)\cap J(f)$ is a {\sw} for all $R>R(f)$ or there exists a dense set of values of $R\in (R(f),\infty)$ such that $A_R(f)\cap J(f)$ has uncountably many components.
\end{theorem}

In the case that~$f$ {\it does} have a {\mcfc}, the sets $I(f)$, $A(f)$ and $A_R(f)$ are all connected, and are in fact spiders' webs \cite[Theorem~1.5]{RS09}. However, in this case the components of $I(f)\cap J(f)$, $A(f)\cap J(f)$ and $A_R(f)\cap J(f)$ are all bounded. The following theorem indicates that in this case the structure of the families of components of the sets $A(f)\cap J(f)$ and $A_R(f)\cap J(f)$ can display strikingly varied behaviour. The term `inner connectivity' used here is explained and discussed in detail in Section~\ref{mcwds}.
\begin{theorem}\label{varied}
(a)\;Let~$f$ be a \tef\ with a {\mcfc}~$U$ and suppose that $R>R(f)$. If $U$ has infinite inner connectivity, then $A_R(f)\cap J(f)$ and $A(f) \cap J(f)$ have uncountably many components.

(b) \;There exists a \tef\ $f$ with a {\mcfc} and $R>R(f)$ such that $A_R(f)\cap J(f)$ and $A(f)\cap J(f)$ have only countably many components.
\end{theorem}

{\it Remark}\; In the course of proving Theorem~\ref{varied} we also obtain new results of independent interest on the possible structures of {\mconn} wandering domains. In particular, we show that, perhaps surprisingly, there exist {\tef}s with {\mcfc}s that have uncountably many complementary components with no interior; see Theorem~\ref{inner}.

Finally, we show that the situation for components of $I(f)\cap J(f)$, in the case when~$f$ has a \mcfc, is more straightforward.
\begin{theorem}\label{mc-IJ}
Let $f$ be a \tef\ with a \mcfc. Then $I(f)\cap J(f)$ has uncountably many components.
\end{theorem}

The plan of the paper is as follows. Section~\ref{Epointconstruction} contains our construction of uncountably many Eremenko points, and the proofs of Theorems~\ref{I(f)} and \ref{A(f)} are then given in Section~\ref{I(f)proof}. Further properties of the Eremenko points construction are given in Section~\ref{furtherprops}, followed by the proof of Theorem~\ref{AR(f)} in Section~\ref{AR(f)proof}. Section~\ref{mcwds} contains background material on {\mcfc}s, Section~\ref{mcwds-bdrycmpnts} gives our new results on the structure of such wandering domains, and Section~\ref{varied-mc-IJ} contains the proofs of Theorems~\ref{varied} and~\ref{mc-IJ}. Note that Sections~\ref{mcwds},~~\ref{mcwds-bdrycmpnts} and~\ref{varied-mc-IJ} can be read independently of the earlier sections. Finally, in Section~\ref{outstanding} we state some open problems related to our results.

\section{Constructing uncountably many Eremenko points}
\setcounter{equation}{0}\label{Epointconstruction}
It was shown in~\cite{BH99} that Eremenko's construction in~\cite{aE89} of points in $I(f)$ actually gives points that are in $A(f)$. Points constructed in this way have particularly nice properties and as noted earlier we often refer to them as {\it Eremenko points}; see \cite{RS09b} and \cite{BRS11}. Eremenko's construction was based on Wiman--Valiron theory, and here we use a modification of this construction to give uncountably many such points.  In Theorem~\ref{WV} we give a key result of Wiman--Valiron theory (see \cite{wH74}, for example) which describes the behaviour of an entire function~$f$ near points at which~$f$ takes its maximum modulus.

Let $f(z)=\sum_{n=0}^{\infty} a_n z^n$ be a {\tef} and, for $r>0$, let $z(r)$ denote a point such that
\[
|z(r)|=r\;\;\text{and}\;\;|f(z(r))| = M(r).
\]
Also, let $N(r)$ be the largest value of $n$ for which $|a_n|r^n$ is maximal. Note that $N(r)$ is increasing with $r$ and $N(r)\to\infty$ as $r\to\infty$.

\begin{theorem}\label{WV}
Let $f$ be a {\tef} and let $\alpha > 1/2$.  There exists a set $E \subset (0,\infty)$ such that $\int_E(1/t)\,dt<\infty$ and, for $r \in (0,\infty)\setminus E$, if $|z-z(r)| < r/(N(r))^{\alpha}$, then
\begin{equation}\label{WVe}
f(z)=\left(\frac{z}{z(r)}\right)^{N(r)} f(z(r))(1+\eps(r,z,\alpha)),
\end{equation}
where $\eps(r,z,\alpha)\to 0$ uniformly with respect to $z$ as $r\to \infty, r\notin
E$.
\end{theorem}
We need the following consequence of Theorem~\ref{WV}, which is related to \cite[Theorem~2.4]{RS09} but gives more precise information.
\begin{theorem}\label{quads}
Let~$f$ be a {\tef}, $K\ge 20\pi$, and put
\begin{equation}\label{DrK}
D_{r,K} = \{z: |z-z(r)| < Kr/N(r)\},\quad r>0.
\end{equation}
Then there exists a set~$E_K \subset (0,\infty)$ with $\int_{E_K}(1/t)\,dt<\infty$ such that, if $r \in [1,\infty) \setminus E_K$, then
\begin{itemize}
\item the disc $D_{r,K}$ contains a closed quadrilateral~$Q_{r,K}$ that can be partitioned into quadrilaterals $Q_{r,K,j}$, where $j\in\Z$, $|j|\le K/(10\pi)$, labelled in anticlockwise order with respect to the origin, such that $z(r)\in Q_{r,K,0}$ and the interior of each $Q_{r,K,j}$ is a univalent preimage under~$f$ of the cut annulus
\[
\{w:\tfrac12 M(r)<|w|<2M(r),|\arg(w/f(z(r)))|<\pi\};
\]
\item if~$B$ is a compact subset of any disc~$D$ in the cut annulus above and $B_{-1}$ is a component of $f^{-1}(B) \cap Q_{r,K}$, then
\begin{equation}\label{size}
\diam\, B \geq c(f,K) N(r) \frac{M(r)}{r}\, \diam\, B_{-1},
\end{equation}
where $c=c(f,K)>0$ is a constant that depends only on~$f$ and~$K$.
\end{itemize}
\end{theorem}
\begin{proof}
Let $\alpha=3/4$ and let $E$ be the corresponding exceptional set defined in Theorem~\ref{WV}. Since $\alpha<1$, it follows from Theorem~\ref{WV} together with the fact that $N(r)\to\infty$ as $r\to\infty$ that we can take $r(f,K)>0$ so large that, for $r\ge r(f,K)$, $r \notin E$, and $z \in D_{r,K}$, we have the linear approximation
\begin{align}\label{logform}
 \log \left(\frac{f(z)}{f(z(r))}\right)& = N(r)\log\left(\frac{z}{z(r)}\right)+\log(1+\eps(r,z))\notag\\
 & = N(r)\left(\frac{z-z(r)}{z(r)}\right)+\eps_1(r,z),
\end{align}
where $|\eps_1(r,z)|\le 1/100$. We can now use Rouch\'e's theorem to deduce from \eqref{logform} that we can also take $r(f,K)>0$ so large that if $r\ge r(f,K)$, $r\notin E$, then the function
\[
g(z)= \log (f(z)/f(z(r)))
\]
maps $D_{r,K}$ univalently, with $g(z(r))=0$, and that $g(D_{r,K})$ contains the disc with centre~0 and radius $\tfrac12 K$, and hence contains the square
\[
\{w: |\Re w|\le \tfrac14 K,|\Im w|\le \tfrac14 K\}.
\]
Now, since $K\ge 20\pi$, this square contains all rectangles of the form
\[
R_j=\{w: |\Re w|< \log 2,|\Im w-2j\pi|< \pi\},   \quad j\in\Z, |j|\le K/(10\pi),
\]
of which there are at least five. Since $|f(z(r))| = M(r)$, we deduce that, for $r\ge r(f,K)$, $r\notin E$, there are at least five quadrilaterals of the form $Q_{r,K,j}=g^{-1}(R_j)$ in $D_{r,K}$ such that $z(r)\in Q_{r,K,0}$,~$f$ acts univalently on $Q_{r,K,j}$, and
\begin{align*}
f(Q_{r,K,j})&=f(z(r))\exp(R_j)\\
&=\{w:\tfrac12 M(r)<|w|<2M(r),|\arg(w/f(z(r)))|<\pi\},
\end{align*}
as required. The union of the closures of these $Q_{r,K,j}$, $|j|\le K/(10\pi)$, forms the required quadrilateral $Q_{r,K}$, which contains $z(r)$ since $z(r)\in Q_{r,K,0}$.

The estimate~\eqref{size} follows by expressing~$f$ in $D_{r,K}$ as in \eqref{logform} and again considering the function in two stages as $f(z)=f(z(r))\exp g(z)$.
\end{proof}
Next we use Theorem~\ref{quads} to construct uncountably many Eremenko points, each with a different `Wiman--Valiron itinerary' determined by a sequence of quadrilaterals to which Theorem~\ref{quads} has been applied successively, and each lying in $A_R(f)$ for the same suitably defined value of~$R$. In the proof of Theorem~\ref{Epoints} we use the following notation for an open annulus:
\[
A(r,R)=\{z:r<|z|<R\},\quad 0<r<R.
\]
\begin{theorem}\label{Epoints}
Let~$f$ be a {\tef}. There exists $R_1(f)\ge R(f)$ such that if $r_0\ge R_1(f)$, then there exist sequences of positive numbers $(r_n)$, complex numbers $(z_n)$, and quadrilaterals  $(Q_n)$, each of which can be partitioned into the union of five quadrilaterals with interiors $Q_{n,j}$, for $j=-2,-1,0,1,2$, labelled in anticlockwise order with respect to the origin, such that, for $n\ge 0$,
\begin{equation}\label{prop2}
\tfrac54 r_n\le |z_n| \le \tfrac74 r_n\;\;\text{and}\;\; r_{n+1}=|f(z_n)|=M(|z_n|),
\end{equation}
\begin{equation}\label{prop1}
z_n\in Q_{n,0}\subset Q_n\subset A(r_n,2r_n),
\end{equation}
and
\begin{equation}\label{prop3}
f \text{ maps } Q_{n,j} \text{ univalently onto } A_{n+1}, \qfor j=-2,-1,0,1,2,
\end{equation}
where
\[
A_{n+1}=\{w:\tfrac12 r_{n+1}<|w|<2r_{n+1},|\arg(w/f(z_n))|<\pi\}.
\]
Furthermore,
\begin{itemize}
\item[(a)]
the sequence $M^{-n}(r_n)$, $n\ge 0$, is strictly increasing and its limit~$R$ satisfies
\begin{equation}\label{Rlim}
r_0 <|z_0|<R<2r_0;
\end{equation}
\item[(b)]
for each sequence of the form $j_n=\pm 1$, $n\ge 0$, there exists a unique point $z(j_n)\in \overline{Q_{0,j_0}}$ with {\it Wiman--Valiron itinerary} $(j_n)_{n\ge 0}$, in the sense that
\begin{equation}\label{prop4}
f^n(z(j_n)) \in \overline{Q_{n,j_n}}, \quad \text{for } n\ge 0,
\end{equation}
and $z(j_n)\in A_R(f)$;
\item[(c)]
if~$f$ has no {\mcfc}s, then each $z(j_n)$ lies in $J(f)$;
\item[(d)] for each sequence of the form $j_n=\pm 1$, $n\ge 0$, and $k\in \N$, $f^k(z(j_n))$ is the unique point in $\overline{Q_{k,j_k}}$ with itinerary $(j_{k+n})_{n\ge 0}$, and $f^k(z(j_n))\in A_{M^k(R)}(f)$.
    \end{itemize}
\end{theorem}
The quadrilaterals $Q_{n,j}$, $j=-2,\ldots,2$, and the cut annulus $A_{n+1}$ are illustrated schematically in Figure~\ref{pic1}.
\vspace*{0.5cm}
\begin{figure}[hbt!]
\begin{center}
\def\svgwidth{1.2\textwidth}

\begingroup%
  \makeatletter%
  \providecommand\color[2][]{%
    \renewcommand\color[2][]{}%
  }%
  \providecommand\transparent[1]{%

    \renewcommand\transparent[1]{}%
  }%
  \providecommand\rotatebox[2]{#2}%
  \ifx\svgwidth\undefined%
    \setlength{\unitlength}{5413.08910476bp}%
    \ifx\svgscale\undefined%
      \relax%
    \else%
      \setlength{\unitlength}{\unitlength * \real{\svgscale}}%
    \fi%
  \else%
    \setlength{\unitlength}{\svgwidth}%
  \fi%
  \global\let\svgwidth\undefined%
  \global\let\svgscale\undefined%
  \makeatother%
  \begin{picture}(1,0.50715027)%
    \put(0.2,0){\includegraphics[width=\unitlength]{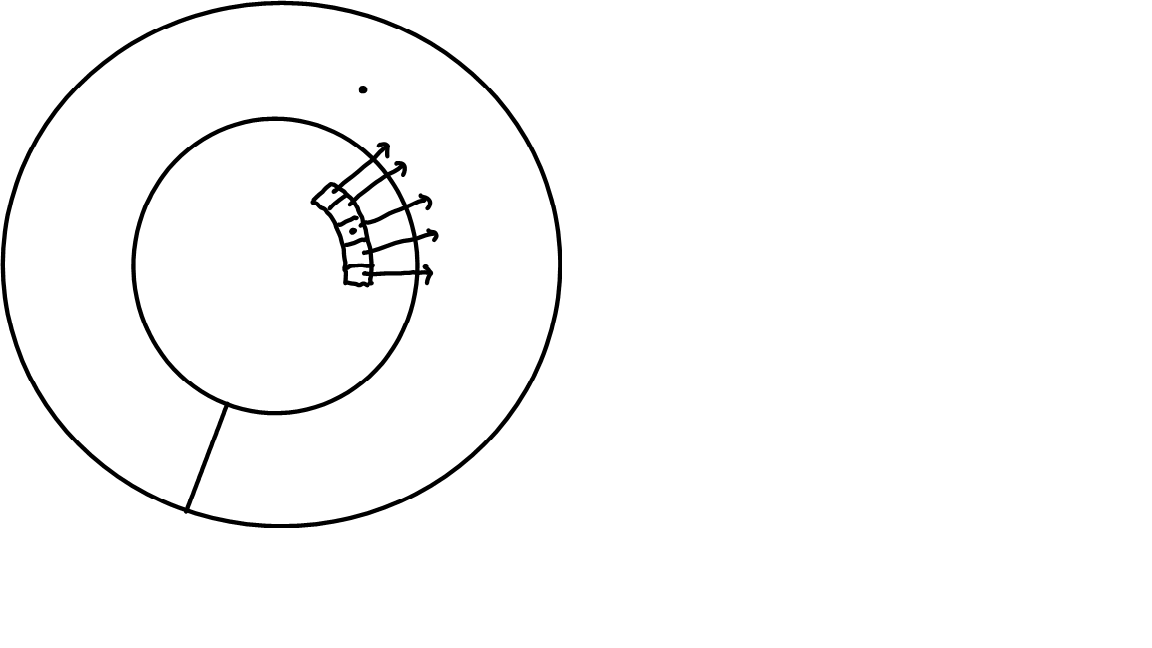}}%
    \put(0.47,0.35){\color[rgb]{0,0,0}\makebox(0,0)[lb]{\smash{$z_n$}}}%
    \put(0.44,0.41){\color[rgb]{0,0,0}\makebox(0,0)[lb]{\smash{$Q_{n,2}$}}}%
    \put(0.48,0.29){\color[rgb]{0,0,0}\makebox(0,0)[lb]{\smash{$Q_{n,-2}$}}}%
    \put(0.525,0.384){\color[rgb]{0,0,0}\makebox(0,0)[lb]{\smash{$f$}}}%
    \put(0.525,0.48){\color[rgb]{0,0,0}\makebox(0,0)[lb]{\smash{$f(z_n)$}}}%
    \put(0.6,0.3){\color[rgb]{0,0,0}\makebox(0,0)[lb]{\smash{$A_{n+1}$}}}%
\end{picture}%
\endgroup%
\end{center}
\vspace*{-1.5cm}
\caption{\small The quadrilaterals $Q_{n,j}$, $j=-2,\ldots,2$, and the cut annulus $A_{n+1}$}
\label{pic1}
\end{figure}
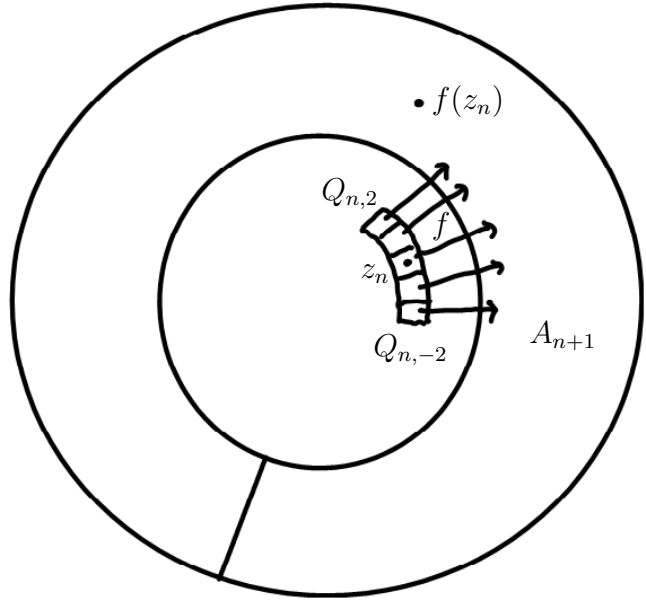

{\it Remarks}\; 1.\;It will be clear from the proof of Theorem~\ref{Epoints} that the constructed Eremenko point $z(j_n)$ actually satisfies $f^n(z(j_n)) \in Q_{n,j_n}$, for $n\ge 0$.

2.\;In this proof, and in later proofs in this paper, we often use the fact that
\begin{equation}\label{equiv}
z\in A_R(f)\;\;\text{ if and only if }\;\;f^k(z)\in A_{M^k(R)}(f),\qfor R>R(f), k\in \N,
\end{equation}
which follows immediately from the definition of $A_R(f)$ in \eqref{ARfdef}.
\begin{proof}[Proof of Theorem~\ref{Epoints}]
We apply Theorem~\ref{quads} with $K=20\pi$ and $E_K$ the corresponding exceptional set given by Theorem~\ref{quads}. Then there exists $R_1(f)>0$ so large that
\begin{equation}\label{PE1}
 \int_{E_K \cap (R_1(f), \infty)} \frac{1}{t}\,dt < \log \frac75,
\end{equation}
and
\begin{equation}\label{PE2}
\frac{K}{N(r)} < \frac18\;\;\text{and}\;\;M(r)>r,\;\;\text{for } r\ge R_1(f).
\end{equation}
As in Theorem~\ref{quads}, we let $D_{r,K} = \{z: |z-z(r)| < Kr/N(r)\}$ and note that, by \eqref{PE2},
\begin{equation}\label{PE3}
D_{r,K}\subset A\left(\tfrac78 r, \tfrac98 r\right),\quad\text{for } r\ge R_1(f).
\end{equation}

Take $r_0\ge R_1(f)$. It follows from Theorem~\ref{quads},~\eqref{PE1} and \eqref{PE2} that there exists
\begin{equation}\label{r0}
r'_0\in \left[\tfrac54 r_0,\tfrac74 r_0\right]\setminus E_K,
\end{equation}
and $z_0=z(r'_0)$ such that
\begin{itemize}
\item
the disc $D_0=D_{r'_0,K}$ contains a quadrilateral~$Q_0$ that can be partitioned into five quadrilaterals $Q_{0,j}, j=-2,-1,0,1,2$, with $z_0\in Q_{0,0}$, the interior of each of which is a univalent preimage under~$f$ of the cut annulus
\[
A_1=\{w:\tfrac12 M(|z_0|)<|w|<2M(|z_0|),|\arg(w/f(z_0))|<\pi\};
\]
\item
if $B$ is a compact subset of any disc $D \subset A_1$ and $B_{-1}$ is a component of $f^{-1}(B)\cap Q_{0}$, then
\[
\diam\, B \geq c(f,K) N(|z_0|) \frac{M(|z_0|)}{|z_0|}\, \diam\, B_{-1},
\]
where $c=c(f,K)>0$ is a constant that depends only on~$f$ and~$K$.
\end{itemize}
Note that, by \eqref{PE3} and \eqref{r0}, we have
\begin{equation}\label{D0}
Q_0\subset D_0\subset A(r_0,2r_0).
\end{equation}
Repeating this process with $r_1=|f(z_0)|=M(|z_0|)$ instead of $r_0$, we deduce
%from Theorem~\ref{quads},~\eqref{PE1} and~\eqref{PE2}
that there exists $r'_1\in \left[\tfrac54 r_1,\tfrac74 r_1\right]\setminus E_K$ and $z_1=z(r'_1)$ such that
\begin{itemize}
\item
the disc $D_1=D_{r'_1,K}$ contains a quadrilateral~$Q_1$ that can be partitioned into five quadrilaterals $Q_{1,j}, j=-2,-1,0,1,2$, with $z_1\in Q_{1,0}$, the interior of each of which is a univalent preimage under~$f$ of the cut annulus
\[
A_2=\{w:\tfrac12 M(|z_1|)<|w|<2M(|z_1|),|\arg(w/f(z_1))|<\pi\};
\]
\item
if $B$ is a compact subset of any disc $D \subset A_2$ and $B_{-1}$ is a component of $f^{-1}(B)\cap Q_{1}$, then
\begin{equation}
\diam\, B \geq c(f,K) N(|z_1|) \frac{M(|z_1|)}{|z_1|}\, \diam\, B_{-1}.\notag
\end{equation}
\end{itemize}

Carrying out this process repeatedly, we obtain sequences of positive numbers $(r_n)$, complex numbers $(z_n)$ such that $\frac54 r_n\le |z_n| \le \frac74 r_n$, and discs
\begin{equation}\label{Dn}
D_n=\{z: |z-z_n| < K|z_n|/N(|z_n|)\}\subset A(r_n,2r_n),\quad n\ge 0,
\end{equation}
quadrilaterals $Q_{n,j}\subset Q_n \subset D_n$, $j=-2,-1,0,1,2$, and cut annuli
\[
A_{n+1}=\{w:\tfrac12 r_{n+1}<|w|<2r_{n+1},|\arg(w/f(z_n))|<\pi\},
\]
that satisfy \eqref{prop2}, \eqref{prop1} and \eqref{prop3}, and also:

if $B$ is a compact subset of any disc $D \subset A_{n+1}$ and $B_{-1}$ is a component of $f^{-1}(B)\cap Q_{n}$, then
\begin{equation}\label{sizen}
\diam\, B \geq c(f,K) N(|z_n|) \frac{M(|z_n|)}{|z_n|}\, \diam\, B_{-1}.
\end{equation}

To prove part~(a) we note that, by the construction,
\[
r_{n+1}=M(|z_n|)>M(r_n)\;\;\text{and}\;\; M(2r_n)>2M(|z_n|)=2r_{n+1},\qfor n\ge 0.
\]
Hence,
\begin{equation}\label{Rineq}
r_0<|z_0|=M^{-1}(r_1)<M^{-2}(r_2)< \cdots <M^{-2}(2r_2)<M^{-1}(2r_1)<2r_0,
\end{equation}
from which part~(a) follows.

Now let $(j_n)$ denote any sequence whose elements are~$\pm 1$. It follows from \eqref{prop1} and \eqref{prop3} that
\[
f(\overline{Q_{n,j_n}})\supset \overline{A_{n+1}} \supset \overline{Q_{n+1,j_{n+1}}},\quad\text{for } n\ge 0.
\]

Thus, given $(j_n)$, we can construct a sequence of compact sets $B_n$ such that $B_0 = \overline{Q_{0,j_0}}$ and,
for $n \in \N$, $B_n$ is a component of $f^{-n}(\overline{Q_{n,j_n}})$ with $B_n \subset B_{n-1}$.
Then $\bigcap_{n=0}^{\infty}B_n$ is a nested intersection of compact sets and is therefore non-empty.

To prove part~(b) we show that $\bigcap_{n=0}^{\infty}B_n$ consists of a single point. By \eqref{sizen}, applied with $B=f^{k+1}(B_n)$ and $B_{-1}=f^k(B_n)$ for $k=0,1,\ldots, n-1$, and also~\eqref{prop2},~\eqref{prop1} and~\eqref{Dn},  we deduce that, for each $n \in \N$,
\begin{align*}
\diam \,B_n \leq & \left(\prod_{k=0}^{n-1} \frac{|z_k|}{c(f,K) N(|z_k|)M(|z_k|)}\right)\diam\, \overline{Q_{n,j_n}} \\
& \leq \left(\prod_{k=0}^{n-1} \frac{|z_k|}{c(f,K) N(|z_k|)M(|z_k|)}\right) \frac{2K|z_n|}{N(|z_n|)} \\
 & \leq \frac{K 2^{n+1}|z_0|}{c(f,K)^{n}N(|z_0|)N(|z_1|)\cdots N(|z_{n}|)} \to 0 \;\mbox{ as } n \to \infty,
\end{align*}
since $N(r)\to \infty$ as $r\to\infty$. Thus $\bigcap_{n=0}^{\infty}B_n$ consists of a single point.

For the given sequence $(j_n)$, we let $\bigcap_{n=0}^{\infty}B_n=\{z(j_n)\}$. Then, for each $n \in \N$, $f^n(z(j_n)) \in \overline{Q_{n,j_n}}$, so \eqref{prop4} holds.

%Now, by~\eqref{prop1} and~\eqref{prop2}, we have
%\begin{equation}\label{fast1}
%|f^n(z(j_n))|> r_n = M(|z_{n-1}|) > M(r_{n-1}), \mbox{ for }  n \in \N.
%\end{equation}
%We deduce that
%\begin{equation}\label{Mnest1}
%|z(j_n)|\ge M^{-n}(|f^n(z(j_n))|)> M^{-n}(r_n)>M^{-n+1}(r_{n-1}), \quad\text{for } n\in\N.
%\end{equation}
%Hence
%\begin{equation}\label{Mnest2}
%R=\lim_{n\to\infty}M^{-n}(r_n)\le |z(j_n)|,
%\end{equation}
%so \eqref{Rlim} holds, by \eqref{D0}.

We now show that $z(j_n)\in A_R(f)$. If not, there exists $R'<R$ and $N_1\in\N$ such that
\[
|f^n(z(j_n))|\le M^n(R'),\quad\text{for } n\ge N_1.
\]
But, by \eqref{Rineq}, there also exists $N_2\in\N$ such that
\[
M^{-n}(r_n)>R',\quad\text{for } n\ge N_2,
\]
so
\[
|f^n(z(j_n))|< r_n,\quad \text{for }n\ge \max\{N_1,N_2\},
\]
a contradiction to \eqref{prop4}. This proves part~(b).

To prove part~(c), we observe that if there exists an open disc~$D$ such that $z(j_n)\in D\subset F(f)$, then there exists $N_0\in\N$ such that
\[
B_n\subset D,\qfor n\ge N_0.
\]
Since $B_n$ is a component of $f^{-n}(\overline{Q_{n,j_n}})$, we deduce that
\[
\overline{Q_{n,j_n}}=f^n(B_n)\subset f^n(D)\subset F(f),\qfor n\ge N_0,
\]
so
\[
\overline{A_{n+1}}= f^{n+1}(B_n)\subset F(f), \qfor n\ge N_0,
\]
and hence $\overline{A_{n+1}}$ is contained in a {\mcfc} for~$n$ sufficiently large.

Finally, to prove part~(d) we follow the construction in part~(a), but start from $Q_{k,j_k}$ instead of $Q_{0,j_0}$, with the itinerary $(j_{k+n})$ instead of $(j_n)$, and use the facts that $f^k(z)\in A_{M^k(R)}(f)$ whenever $z\in A_R(f)$, and
\[
\lim_{n\to\infty}M^{-n}(r_{k+n})=\lim_{n\to\infty}M^k(M^{-n-k}(r_{k+n}))=M^k(R).\qedhere
\]
\end{proof}

\section{Proofs of Theorems~\ref{I(f)} and~\ref{A(f)}}
\setcounter{equation}{0}\label{I(f)proof}
To prove Theorem~\ref{I(f)}, we need some results from plane topology, which we state for the reader's convenience. The results in the following lemma are classical; see \cite[pages 84~and~143]{New}.

 \begin{lemma}\label{Newman}
 \begin{itemize}
 \item[(a)] If $E_0$ is a continuum in $\hat{\C}$, $E_1$ is a closed subset of $E_0$ and $C$ is a component of
 $E_0\setminus E_1$, then $\overline{C}$ meets $E_1$.
 \item[(b)] If $C_1$ and $C_2$ are two components of a closed set $E$ in $\hat{\C}$, then there is a Jordan curve
 in $\hat{\C}\setminus E$ that separates $C_1$ and $C_2$.
 \end{itemize}
 \end{lemma}
Lemma~\ref{Newman}\,(a) has the following corollary, which we use frequently.
\begin{corollary}\label{Newman-cor}
Let $\Gamma$ be an {\ucm} which meets the circle $C=\{z:|z|=r\}$, where $r>0$. Then $\Gamma\cap \{z:|z|\ge r\}$ has at least one component that is an  {\ucm}, $\Gamma'$ say, and any such component satisfies $\Gamma'\cap C\ne\emptyset$.
\end{corollary}
\begin{proof}
The set $\hat \Gamma=\Gamma\cup \{\infty\}$ is a continuum in $\hat \C$. Let $\hat\Gamma'$ be a component of $\hat \Gamma\setminus \{z:|z|\le r\}$ that contains $\infty$. Then the closure of $\hat \Gamma'$ meets~$C$, by Lemma~\ref{Newman}\,(a). Hence, by Lemma~\ref{Newman}\,(a) again,  we can take the required set $\Gamma'$ to be $\hat\Gamma'$ with $\infty$ removed.
\end{proof}
The proof of Theorem~\ref{I(f)} uses the sequences $(r_n)$, $(z_n)$, $(Q_n)$ and $(A_{n+1})$, $n\ge 0$, and the radius $R>0$, which were defined in Theorem~\ref{Epoints}. Also, we put
\[
C_n=\{z:|z|=2r_n\},\quad n\ge 0,
\]
which, for $n\ge 1$, is a subset of the boundary of the cut annulus~$A_n$.

Recall that the set $Q_n$, $n\ge 0$, consists of five quadrilaterals
\[
Q_{n,j},\; j=-2,\ldots,2,
\]
arranged in anticlockwise order around the origin, each of which is mapped one-one and conformally by~$f$ onto the cut annulus $A_{n+1}$. We also put
\[
E_n=\partial Q_n \cap f^{-1}(C_{n+1}),
\]
which is the outer edge of the quadrilateral $Q_n$.

The following concept is fundamental to our proof. An {\it escape channel} at level~$n$, $n\ge 0$, is a triple $\Sigma=(\Gamma^-, \FC, \Gamma^+)$, where $\Gamma^-, \FC, \Gamma^+$ are disjoint {\uca} such that
\begin{itemize}
\item[(a)] $\Gamma^-$, $\FC$ and $\Gamma^+$ all lie in $\{z:|z|\ge r_n\}$ and meet the circle $C_n$ in closed sets that include points of the form $r_ne^{i\theta^-}$, $r_ne^{i\theta}$ and $r_ne^{i\theta^+}$, respectively, where $\theta^-<\theta<\theta^+<\theta^-+2\pi$;
\item[(b)] $\Gamma^-\cup \Gamma^+ \subset I(f)^c$ and $\FC\subset A_{M^n(R)}(f)$.
\end{itemize}

The {\it interior} of an $n$-th level escape channel $\Sigma=(\Gamma^-, \FC, \Gamma^+)$ is the complementary component of $\Gamma^-\cup C_n\cup \Gamma^+$ that contains $\FC\setminus C_n$. Two $n$-th level escape channels are called {\it disjoint} if their interiors are disjoint. For any $n$-th level escape channel~$\Sigma$ the {\it entry} to~$\Sigma$ is the largest closed arc of $C_n$ that forms part of the boundary of the interior of~$\Sigma$. If the entry of one $n$-th level escape channel~$\Sigma$ is a subset of the entry of another $n$-th level escape channel~$\Sigma'$, then we write $\Sigma\prec \Sigma'$.

The next lemma shows that any $n$-th level escape channel can be pulled back under~$f$ and truncated to produce several disjoint $(n-1)$-th level escape channels.

\begin{lemma}\label{ucc-pullback}
Let $(\Gamma^-, \FC, \Gamma^+)$ be an $n$-th level escape channel, where $n\ge 1$. Then there exist four disjoint $(n-1)$-th level escape channels, $(\Gamma_k^-, \FC_k, \Gamma_k^+)$, $k=0,\ldots,3$, such that
\begin{equation}\label{channels}
f(\Gamma_k^-)\subset \Gamma^-,\;\;f(\FC_k)\subset \FC,\;\;f(\Gamma_k^+)\subset \Gamma^+,\qfor k=0,\ldots,3.
\end{equation}
\end{lemma}
\begin{proof}
Since~$f$ maps each quadrilateral $Q_{n-1,j}$, $j=-2,\ldots,2$, univalently onto the cut annulus $A_n$, we deduce that there are at least four distinct triples of preimage components of $(\Gamma^-,\FC,\Gamma^+)$ that meet $E_{n-1}$ in disjoint compact sets, which lie in order anticlockwise along $E_{n-1}$; see Figure~\ref{pic2}. All these preimage sets are unbounded, by Lemma~\ref{Newman}\,(b), since~$f$ is an open map, and none of them have any points in the interior of $Q_{n-1}$. The preimage components of $\FC$ are all in $A_{M^{n-1}(R)}(f)$ and the preimage components of $\Gamma^-$ and $\Gamma^+$ are all in $I(f)^c$.
\vspace*{0.5cm}
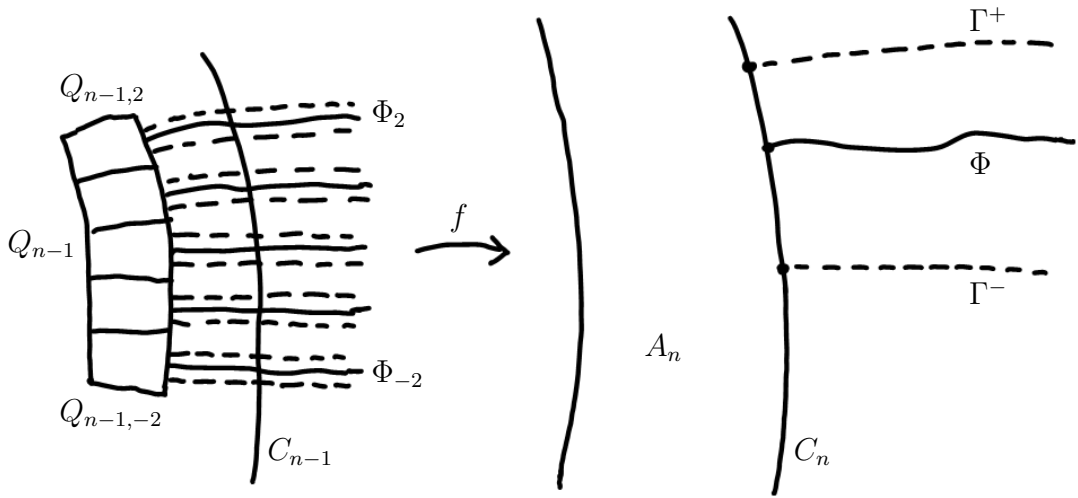
\begin{figure}[hbt!]
\begin{center}
\def\svgwidth{1.2\textwidth}

\begingroup%
  \makeatletter%
  \providecommand\color[2][]{%
    \renewcommand\color[2][]{}%
  }%
  \providecommand\transparent[1]{%

    \renewcommand\transparent[1]{}%
  }%
  \providecommand\rotatebox[2]{#2}%
  \ifx\svgwidth\undefined%
    \setlength{\unitlength}{5413.08910476bp}%
    \ifx\svgscale\undefined%
      \relax%
    \else%
      \setlength{\unitlength}{\unitlength * \real{\svgscale}}%
    \fi%
  \else%
    \setlength{\unitlength}{\svgwidth}%
  \fi%
  \global\let\svgwidth\undefined%
  \global\let\svgscale\undefined%
  \makeatother%
  \begin{picture}(1,0.50715027)%
    \put(0.05,0){\includegraphics[width=\unitlength]{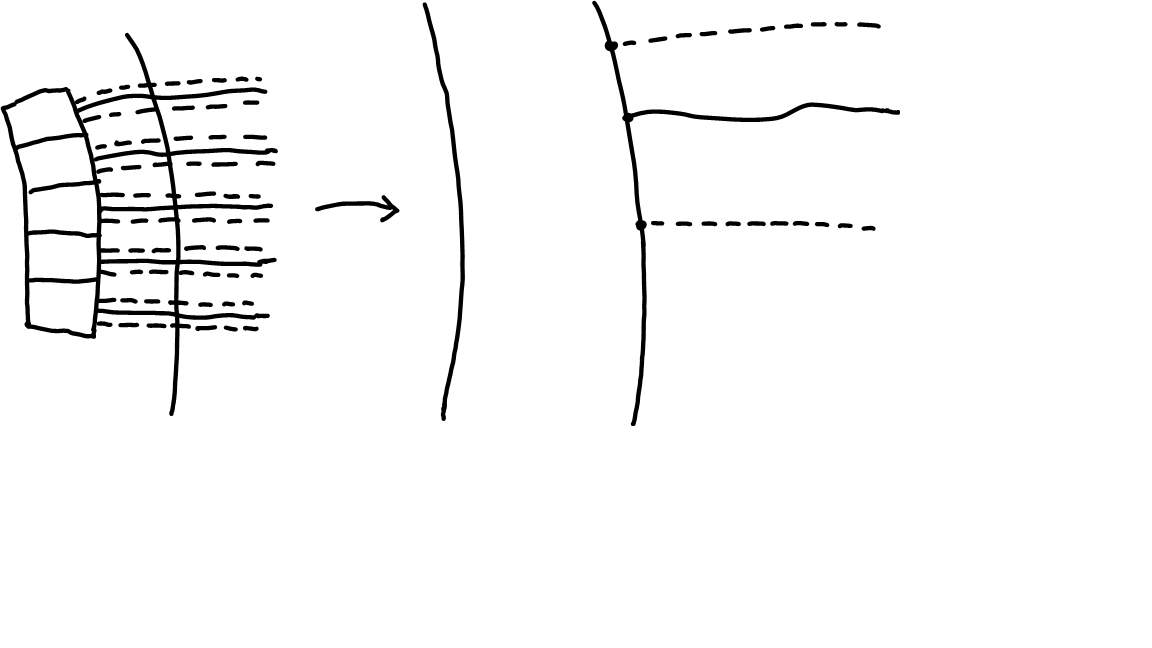}}%
    \put(0.01,0.38){\color[rgb]{0,0,0}\makebox(0,0)[lb]{\smash{$Q_{n-1}$}}}%
    \put(0.05,0.5){\color[rgb]{0,0,0}\makebox(0,0)[lb]{\smash{$Q_{n-1,2}$}}}%
    \put(0.05,0.25){\color[rgb]{0,0,0}\makebox(0,0)[lb]{\smash{$Q_{n-1,-2}$}}}%
    \put(0.29,0.48){\color[rgb]{0,0,0}\makebox(0,0)[lb]{\smash{$\Phi_{2}$}}}%
    \put(0.29,0.28){\color[rgb]{0,0,0}\makebox(0,0)[lb]{\smash{$\Phi_{-2}$}}}%
    \put(0.35,0.4){\color[rgb]{0,0,0}\makebox(0,0)[lb]{\smash{$f$}}}%
    \put(0.75,0.34){\color[rgb]{0,0,0}\makebox(0,0)[lb]{\smash{$\Gamma^-$}}}%
    \put(0.75,0.44){\color[rgb]{0,0,0}\makebox(0,0)[lb]{\smash{$\Phi$}}}%
    \put(0.75,0.55){\color[rgb]{0,0,0}\makebox(0,0)[lb]{\smash{$\Gamma^+$}}}%
    \put(0.5,0.3){\color[rgb]{0,0,0}\makebox(0,0)[lb]{\smash{$A_{n}$}}}%
    \put(0.615,0.22){\color[rgb]{0,0,0}\makebox(0,0)[lb]{\smash{$C_{n}$}}}%
    \put(0.21,0.22){\color[rgb]{0,0,0}\makebox(0,0)[lb]{\smash{$C_{n-1}$}}}%
\end{picture}%
\endgroup%
\end{center}
\vspace*{-3.5cm}
\caption{\small Preimages of an $n$-th level escape channel}
\label{pic2}
\end{figure}

We choose four of these triples, which consist of twelve {\uca}, all pairwise disjoint, and for each of these {\uca} we take an {\ucc} subset of the intersection of the set with $\{z:|z|\ge 2r_{n-1}\}$ that meets $C_{n-1}$. This is possible by Corollary~\ref{Newman-cor}. These twelve subsets lie in the same order (around $C_{n-1}$) as do the twelve {\uca} comprising the four triples (along $E_{n-1}$). Hence these twelve subsets form four disjoint escape channels at level~$n-1$, which can be denoted by $(\Gamma_k^-, \FC_k, \Gamma_k^+)$, $k=0,\ldots,3$, and with this notation it follows by the construction that \eqref{channels} holds.
\end{proof}

We now give the proof of Theorem~\ref{I(f)}.

\begin{proof}[Proof of Theorem~\ref{I(f)}] Let $\Gamma$ be an {\ucm} lying entirely in $I(f)^c$. Without loss of generality we may suppose that $0\in \Gamma$. Then, for $n\ge 0$, we let $\Gamma_0(n,n)$ denote an unbounded component of $\Gamma\cap \{z:|z|\ge 2r_n\}$ that meets~$C_n$, which is possible by Corollary~\ref{Newman-cor}. The reason for the notation $\Gamma_0(n,n)$ will become clear shortly.

We also introduce a single component, $\FC_0$ say, of $A_R(f)$ which contains an Eremenko point in the quadrilateral $\overline{Q_{0,0}}$, as constructed in Theorem~\ref{Epoints}. For $n\ge 1$, the set $\FC_n=f^n(\FC_0)$ is a component of $A_{M^n(R)}(f)$ and $\FC_n\cap Q_n\ne \emptyset$, by Theorem~\ref{Epoints}\;(a), so we have $\FC_n\cap C_n \ne \emptyset$. Then let $\FC_0(n,n)$ denote an unbounded component of $\FC_n\cap \{z:|z|\ge 2r_n\}$ that meets~$C_n$, which is possible by Corollary~\ref{Newman-cor} again.

The triple $(\Gamma_0(n,n), \FC_0(n,n), \Gamma_0(n,n))$ can be thought of as a degenerate escape channel at level~$n$, for which the two {\uca} in $I(f)^c$ are identical. The proof of Lemma~\ref{ucc-pullback} can readily be adapted to show that there are four disjoint $(n-1)$-th level escape channels,
\[
(\Gamma_k^-(n-1,n), \FC_k(n-1,n), \Gamma_k^+(n-1,n)), \quad k=0,\ldots,3,
\]
such that
\begin{equation}\label{channels-n}
f(\Gamma_k^{\pm}(n-1,n))\subset \Gamma_0(n,n),\;\;f(\FC_k(n-1,n))\subset \FC_0(n,n),\qfor k=0,\ldots,3.
\end{equation}
We can now choose two of these four $(n-1)$-th level escape channels with the additional property that the interior of neither of these escape channels meets~$\Gamma$. This is possible since $\Gamma\subset I(f)^c$ and $\FC_k(n-1,n)\subset A_{M^{n-1}(R)}(f)$ for $k=0,\ldots,3$. We then relabel these two chosen escape channels as
\[
(\Gamma_k^-(n-1,n), \FC_k(n-1,n), \Gamma_k^+(n-1,n)), \quad k=0,1,
\]
and note that \eqref{channels-n} remains true.

We now apply Lemma~\ref{ucc-pullback} in this way repeatedly to produce, for all $n\ge 1$ and $0\le m<n$, a set of $2^{n-m}$ escape channels at level~$m$, denoted by
\[
\Sigma_k(m,n)=(\Gamma_k^-(m,n), \FC_k(m,n), \Gamma^+(m,n)),\quad k=0,\ldots, 2^{n-m}-1,
\]
such that
\[
f(\Gamma_k^{\pm}(m,n))\subset \Gamma_{[k/2]}^{\pm}(m+1,n),\;\;f(\FC_k(m,n))\subset \FC_{[k/2]}(m+1,n),\qfor k=0,\ldots,2^{n-m}-1,
\]
and, in addition,
\[
\Sigma_k(m,n)\prec \Sigma_{[k/2]}(m,n-1),\qfor k=0,\ldots,2^{n-m}-1.
\]
Hence we have, for $m=0$, by induction,
\begin{equation}\label{level0-1}
\Gamma_k^{\pm}(0,n)\subset I(f)^c,\;\; \FC_k(0,n)\subset A_R(f),\qfor k=0,\ldots, 2^n-1,
\end{equation}
and
\begin{equation}\label{level0-2}
\Sigma_k(0,n)\prec \Sigma_{[k/2]}(0,n-1),\qfor k=0,\ldots, 2^n-1,
\end{equation}
For $n\ge 1$ and $k=0,\ldots, 2^n-1$, we let $\sigma_k(0,n)$ denote the entry to the channel $\Sigma_k(0,n)$. Then each $\sigma_k(0,n)$ is a closed arc of $C_0$ which has endpoints in $\Gamma_k^{\pm}(0,n)\subset I(f)^c$ and contains a point of $\FC_k(0,n)\subset A_R(f)$, by \eqref{level0-1}. Also,
\begin{equation}\label{sigma}
\sigma_k(0,n)\subset \sigma_{[k/2]}(0,n-1),\qfor k=0,\ldots, 2^n-1,
\end{equation}
by \eqref{level0-2}. Now let
\[
S_n=\bigcup_{k=0}^{2^n-1} \sigma_k(0,n),\;\;n=1,2,\ldots.
\]
Then $(S_n)$ is a nested sequence of compact subsets of $C_0$, whose intersection~$S$ has, by \eqref{sigma}, uncountably many components, and all but at most countably many of these must be singletons.

Let $\{\zeta\}$ be a singleton component of~$S$. Then we deduce that there is a sequence of points $\zeta_n\in C_0$ and integers $k(n)$, for $n\ge 1$, such that
\begin{equation}\label{nested2}
\zeta_n\in \sigma_{k(n)}(0,n)\cap \FC_{k(n)}(0,n)\qfor n\ge 1,\;\;\text{and}\;\; \{\zeta\}=\bigcap_{n=1}^{\infty}\sigma_{k(n)}(0,n).
\end{equation}
In particular, $\zeta_n \to \zeta$ as $n\to\infty$.

Then $\hat \FC_{k(n)}(0,n)=\FC_{k(n)}(0,n)\cup \{\infty\}$, $n\ge 0$, forms a sequence of continua in $\hat\C$, so we may assume, by taking a subsequence if necessary, that $\hat \FC_{k(n)}(0,n)$ converges with respect to the Hausdorff metric on $\hat\C$ to a continuum in $\hat\C$ containing~$\infty$ and~$\zeta$; see \cite[pages 37--39]{kF89}. Since $A_R(f)$ is closed, the part of this limiting continuum in $\C$ is contained in $A_R(f)$. Hence $\zeta$ lies in an {\ucc} subset of $A_R(f)$, which we denote by $\FC_{\zeta}$. We denote by $I_{\zeta}$ the component of $I(f)\cap \{z:|z|\ge 2r_0\}$ that contains $\FC_{\zeta}$.

Suppose now that $\zeta$ and $\zeta'$ are distinct singleton components of~$S$. Then we claim that $I_{\zeta}$ and $I_{\zeta'}$ are disjoint. Indeed, it follows by \eqref{nested2} that there are {\uca} in $I(f)^c$, which lie in $\{z:|z|\ge 2r_0\}$ and which meet $C_0$ at points that lie on either side of~$\zeta$ and as close as we like to~$\zeta$. This proves our claim.

To summarise what we have proved, the set $I(f)\cap \{z:|z| \ge 2r_0\}$ has uncountably many components $I_{\zeta}$, $\zeta\in S$, each of which contains an {\ucm} $\FC_{\zeta}$ such that $\zeta\in \FC_{\zeta}\subset A_R(f)$. Moreover, for any two distinct singleton components of~$S$, $\{\zeta\}$ and $\{\zeta'\}$, there are {\uca} in $I(f)^c$, which lie in $\{z:|z|\ge 2r_0\}$ and separate $I_{\zeta}$ from $I_{\zeta'}$.

Now let~$D$ be any open disc that meets $J(f)$. Then, by the blowing up property of $J(f)$, there exists~$N\in\N$ such that
\[
D_n=\wt{f^n(D)} \supset \{z:|z|\le 2r_0\},\qfor n\ge N.
\]
Here the notation $\wt{U}$ denotes the union of the set~$U$ with all its bounded complementary components. By what we proved above, the set $\C\setminus D_N$ contains uncountably many components of $I(f)\setminus D_N$, each meeting $\partial D_N$ and containing an {\ucm} in $A_R(f)$, and each pair of these components is separated in $I(f)\setminus D_N$ by an {\ucm} in $I(f)^c$.

Since $\partial D_N\subset \partial f^N(D)\subset f^N(\partial D)$, we deduce that there is an arc $\alpha$ of $\partial D$ such that $f^N(\alpha)$ is an arc of $\partial D_N$ that contains points of uncountably many components of $I(f)\setminus D_N$ each containing an {\ucm} in $A_R(f)$, and each pair of which is separated in $\C\setminus D_N$ by an {\ucm} in $I(f)^c$ that meets $f^N(\alpha)$.

Hence, by another application of Lemma~\ref{Newman}, there are uncountably many components of $I(f)\setminus D$ each containing an {\ucm} in $f^{-N}(A_R(f))\subset A(f)\subset I(f)$, and each pair of which is separated in $\C\setminus D$ by an {\ucm} in $I(f)^c$ that meets the arc~$\alpha$. This completes the proof of Theorem~\ref{I(f)}.
\end{proof}

The analogous result Theorem~\ref{A(f)}, concerning the set $A(f)$, has a proof that is identical to the proof of Theorem~\ref{I(f)}, with $I(f)$ replaced by $A(f)$ throughout, and we omit the details.

\section{Further properties of the Eremenko points construction}
\setcounter{equation}{0}\label{furtherprops}
Theorem~\ref{Epoints} shows that in each interval of the form $(r_0,2r_0)$, where $r_0\ge R_1(f)$, we can choose~$R$ with the property that there are points in $A_R(f)$ each of whose orbits passes through a sequence of quadrilaterals $Q_{n,j_n}\subset Q_n$, $n\ge 0$, corresponding to one of the uncountably many Wiman--Valiron itineraries $(j_n)$, $j_n=\pm 1$. To prove Theorem~\ref{AR(f)} we shall require some further properties of these quadrilaterals $Q_{n,j_n}$, each of which is a univalent preimage under~$f$ of
\[
A_{n+1}=\{w:\tfrac12 r_{n+1}<|w|<2r_{n+1},|\arg(w/f(z_n))|<\pi\}.
\]
We label the `inner edges' of $Q_n$ and $Q_{n,j_n}$, which are mapped under~$f$ to $\{w:|w|=\tfrac12 r_{n+1}\}$, as $\alpha_n$ and $\alpha_{n,j_n}$, respectively.
\begin{theorem}\label{further}
Let $f$ be a {\tef}, let $(r_n)$, $(z_n)$, $(Q_n)$, $(Q_{n,j_n})$,~$z(j_n)$ and~$R$ be as in Theorem~\ref{Epoints}, and let the inner edges of $Q_n$ and $Q_{n,j_n}$ be as defined above. Then, for $n\ge 0$,
\begin{equation}\label{ineq1}
r_n<|z_n|< M^n(R)\le |f^n(z(j_n))|<2r_n,
\end{equation}
and
\begin{equation}\label{ineq2}
\alpha_n \subset \{z:|z|<|z_n|\}\subset \{z:|z|<M^n(R)\}.
\end{equation}

Also, if $G_n$ is the component of $Q_{n,j_n}\setminus A_{M^n(R)}(f)$ whose boundary contains~$\alpha_{n,j_n}$, then $G_n$ is simply connected,
\begin{equation}\label{ineq3}
Q_{n,j_n}\cap\{z:|z|<M^n(R)\} \subset G_n\;\;\text{and}\;\; f^n(z(j_n))\in \overline{G_n}.
\end{equation}
\end{theorem}
\begin{proof}
The inequalities in \eqref{ineq1} follow from the construction in Theorem~\ref{Epoints} in the same way that those in part~(a) of Theorem~\ref{Epoints} do (this is the special case when $n=0$), together with the fact (see Theorem~\ref{Epoints}~(b)) that $z(j_n)\in A_R(f)$.

To prove property \eqref{ineq2} note that, by \eqref{logform},
\begin{equation}\label{gn}
g_n(z)=\log \left(\frac{f(z)}{f(z_n)}\right) = N(|z_n|)\left(\frac{z-z_n}{z_n}\right)+\eps_1(z),\quad\text{for }z\in D_n,
\end{equation}
where the disc $D_n$ is defined by \eqref{Dn} and $|\eps_1(z)|\le 1/100$. The function $g_n$ maps $D_n$ univalently, with $g_n(z_n)=0$, and $g_n(D_n)$ contains the square
\[
\{w: |\Re w|\le 5\pi,|\Im w|\le 5\pi\},
\]
since $K=20\pi$ in Theorem~\ref{Epoints}. Also,
\[
Q_{n,j}=g_n^{-1}(\{w: |\Re w|< \log 2,|\Im w-2j\pi|< \pi\}),   \quad j=-2,-1,0,1,2.
\]
Hence, by the linear approximation \eqref{gn}, we deduce that
\[\alpha_n\subset \{z:|z|<|z_n|\}\subset \{z:|z|<M^n(R)\},\]
as required.

The domain $G_n$ is simply connected because the set $A_{M^n(R)}(f)$ has no bounded components. The first part of \eqref{ineq3} follows from \eqref{ineq1} and \eqref{ineq2}, together with the facts that $f^n(z(j_n))\in \ol{Q_{n,j_n}}$ and $A_{M^n(R)}(f)\subset \{z:|z|\ge M^n(R)\}$.

To prove the second part of \eqref{ineq3}, note first that~$f$ maps $G_n$ univalently onto a simply connected domain whose boundary is contained in $\partial A_{n+1}\cup A_{M^{n+1}(R)}(f)$, by \eqref{equiv}, and whose outer boundary component surrounds the circle $\{w:|w|=M^{n+1}(R)\}$ and so surrounds $\alpha_{n+1}$, by \eqref{ineq2}. Therefore $f(G_n)$ contains the domain $G_{n+1}$ that is the component of $Q_{n+1,j_{n+1}}\setminus A_{M^{n+1}(R)}(f)$ whose boundary contains $\alpha_{n+1,j_{n+1}}$.

Repeating this process, we obtain a sequence of domains $G_{m}\subset Q_{m,j_m}$, $m\ge n$, such that
\[
f(\overline{G_m})\supset \overline{G_{m+1}},\quad m\ge n.
\]
Hence $\overline{G_n}$ contains a point $z'$ such that
\[
f^{m-n}(z')\in \overline{G_{m}}\subset \overline{Q_{m,j_m}},\quad \text{for } m\ge n,
\]
by Lemma~\ref{toplemma}, which we use several times later in the paper. It now follows from the final statement of Theorem~\ref{Epoints} that $z'=f^n(z(j_n))$, so $f^n(z(j_n))\in \overline{G_n}$. This completes the proof of Theorem~\ref{further}.
\end{proof}

\section{Proof of Theorem~\ref{AR(f)}}
\setcounter{equation}{0}\label{AR(f)proof}
To prove Theorem~\ref{AR(f)}, we shall suppose that for every $R>R(f)$ the set $A_R(f)$ is {\it not} a {\sw} and deduce that there exists a dense set of values of $R\in (R(f),\infty)$ such that $A_R(f)$ has uncountably many components. First, recall that for all $R>R(f)$, the set $A_R(f)$ has the property that each of its components is closed and unbounded, and lies in $\{z:|z|\ge R\}$; see \cite[Theorem~1.1]{RS09}.

Now suppose that $r_0\ge R_1(f)$, where $R_1(f)$ is as in Theorem~\ref{Epoints}, and let~$R$ be given by Theorem~\ref{Epoints}~(a), so $r_0<R<2r_0$. Each of the uncountably many Eremenko points $z=z(j_n)$ found in Theorem~\ref{Epoints} lies in an unbounded closed component, $\Gamma(j_n)$ say, of $A_R(f)$. We shall show that the components $\Gamma(j_n)$ are pairwise disjoint.

Suppose that two Eremenko points~$z$ and~$z'$ have different Wiman--Valiron itineraries $(j_n)$ and $(j'_n)$, respectively. Then there exists $N\ge 0$ such that $f^N(z)$ and $f^N(z')$ lie in distinct quadrilaterals $Q_{N,j_N}$ and $Q_{N,j'_N}$, respectively. Both $f^N(z)$ and $f^N(z')$ lie in $A_{M^N(R)}(f)$, by \eqref{equiv}, and if we can show that they lie in distinct components of $A_{M^N(R)}(f)$, then it follows that~$z$ and~$z'$ lie in distinct components of $A_R(f)$.

Without loss of generality we can assume that $N=0$. We can also assume by relabelling that $z\in Q_{0,-1}\subset Q_0$ and $z'\in Q_{0,1}\subset Q_0$. For simplicity, we write $Q=Q_{0,-1}$ and $Q'=Q_{0,1}$, and denote the inner edges of~$Q$ and~$Q'$ by~$\alpha$ and~$\alpha'$, respectively. We know from the proof of Theorem~\ref{Epoints} that~$f$ maps both~$Q$ and~$Q'$ conformally onto a cut annulus of the form
\[
A=\{w:\tfrac12 M(|z_0|)<|w|<2M(|z_0|), |\arg (w/f(z_0))|<\pi\},
\]
where $z_0\in Q_{0,0}$, and $f$ maps both $\alpha$ and $\alpha'$ onto the inner boundary component of~$\overline A$.

Then let~$G$ denote the component of $Q\setminus A_R(f)$ whose boundary contains $\alpha$ and let~$G'$ denote the component of $Q'\setminus A_R(f)$ whose boundary contains~$\alpha'$. We have
\[
\alpha\cup\alpha'\subset\{z:|z|<R\},\;\; Q\cap \{z:|z|<R\}\subset G\;\;\text{and}\;\;Q'\cap \{z:|z|<R\}\subset G',
\]
by \eqref{ineq2} and \eqref{ineq3} with $n=0$. Both~$G$ and~$G'$ are simply connected, and are mapped by~$f$ conformally onto the component,~$H$ say, of $A\setminus A_{M(R)}(f)$ whose boundary contains $\{z: |z|=\tfrac12 r_1\}$, where $r_1=M(|z_0|)$.

By the final statement of Theorem~\ref{further}, in \eqref{ineq3}, we have $z\in \overline{G}$ and $z'\in \overline{G'}$. Now we take a simple path $\gamma\subset G$ for which $\overline{\gamma}\setminus \gamma$ is the union of a point, $z_{\alpha}$ say, that lies on the edge~$\alpha$ and a continuum in $\partial G$ that contains~$z$ but does not contain any open arcs of $\partial Q \setminus A_R(f)$. We can obtain such a path by, for example, using a Riemann mapping of~$G$ onto the open unit disc~$\D$. Under this mapping, the prime end of~$G$ whose impression contains~$z$ corresponds to a point $\zeta_z\in\partial \D$ and $z_{\alpha}$ corresponds to a point $\zeta_{z_{\alpha}}\in \partial\D$. Then we can take~$\gamma$ to be the preimage in~$G$ of the path in~$\D$ consisting of the two radii from~0 to $\zeta_z$ and from~0 to $\zeta_{z_{\alpha}}$; see \cite{Pomm} for the theory of prime ends and, in particular, Carath\'eodory's theorem giving the correspondence between prime ends and boundary points of the open unit disc.

Similarly, take a simple path $\gamma'\subset G'$ for which $\overline{\gamma'}\setminus \gamma'$ is the union of a point on~$\alpha'$ and a continuum in $\partial G'$ that contains~$z'$ but does not contain any open arcs of $\partial Q' \setminus A_R(f)$. Then let $\Gamma$ denote the union of the paths~$\gamma$, $\gamma'$ and the segment of the inner edge of $Q_0$ that joins the endpoints of these two paths in~$\alpha$ and~$\alpha'$.

We now suppose that~$z$ and~$z'$ lie in the same component of $A_R(f)$, say~$\Delta$, and show that this leads to a contradiction. Then
\[
z,z'\in \overline{\Gamma}\setminus \Gamma\subset A_R(f),\quad \text{so}\quad \ol{\Gamma}\setminus \Gamma\subset \Delta.
\]
The statement here that $\overline{\Gamma}\setminus \Gamma\subset A_R(f)$ is true because $A_R(f)$ is closed and $\Gamma$ does not accumulate at any points of $\partial G$ or $\partial G'$ that are outside $A_R(f)$.

We consider the bounded domain~$\Omega$ whose boundary consists of~$\Gamma$ and a subset of~$\Delta$. Since~$\Delta\subset \{z:|z|\ge R\}$, we deduce by \eqref{ineq2} and \eqref{ineq3} that at least one of the following must be the case:
\[
Q_{0,0}\cap \{z:|z|<|z_0|\}\subset \Omega\;\;\text{or}\;\; Q_{0,2}\cap \{z:|z|<|z_0|\}\subset \Omega,
\]
depending on the location of the component~$\Delta$. The two possibilities are illustrated in Figure~\ref{pic3}.
%\vspace*{1cm}
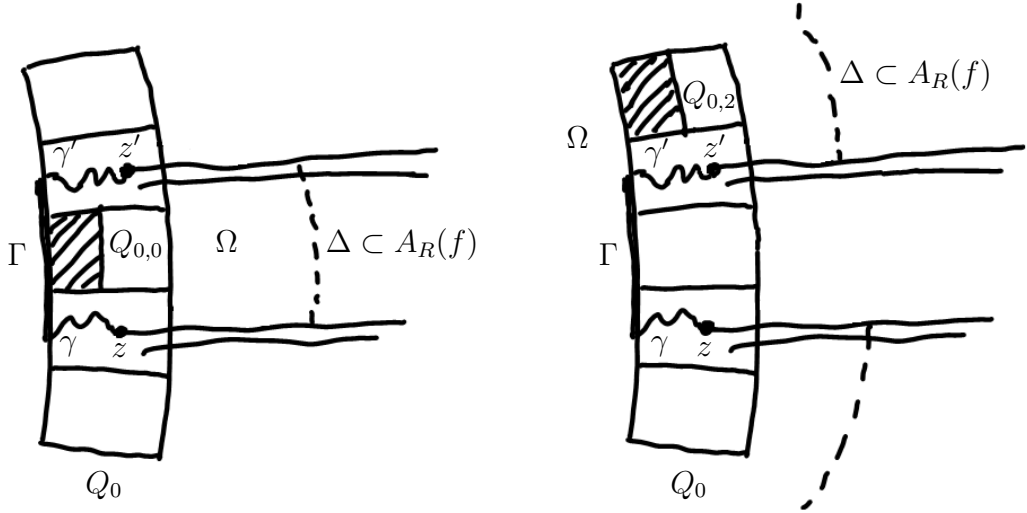
\begin{figure}[hbt!]
\begin{center}
\def\svgwidth{1.2\textwidth}

\begingroup%
  \makeatletter%
  \providecommand\color[2][]{%
    \renewcommand\color[2][]{}%
  }%
  \providecommand\transparent[1]{%

    \renewcommand\transparent[1]{}%
  }%
  \providecommand\rotatebox[2]{#2}%
  \ifx\svgwidth\undefined%
    \setlength{\unitlength}{5413.08910476bp}%
    \ifx\svgscale\undefined%
      \relax%
    \else%
      \setlength{\unitlength}{\unitlength * \real{\svgscale}}%
    \fi%
  \else%
    \setlength{\unitlength}{\svgwidth}%
  \fi%
  \global\let\svgwidth\undefined%
  \global\let\svgscale\undefined%
  \makeatother%
  \begin{picture}(1,0.50715027)%
    \put(0.05,0){\includegraphics[width=\unitlength]{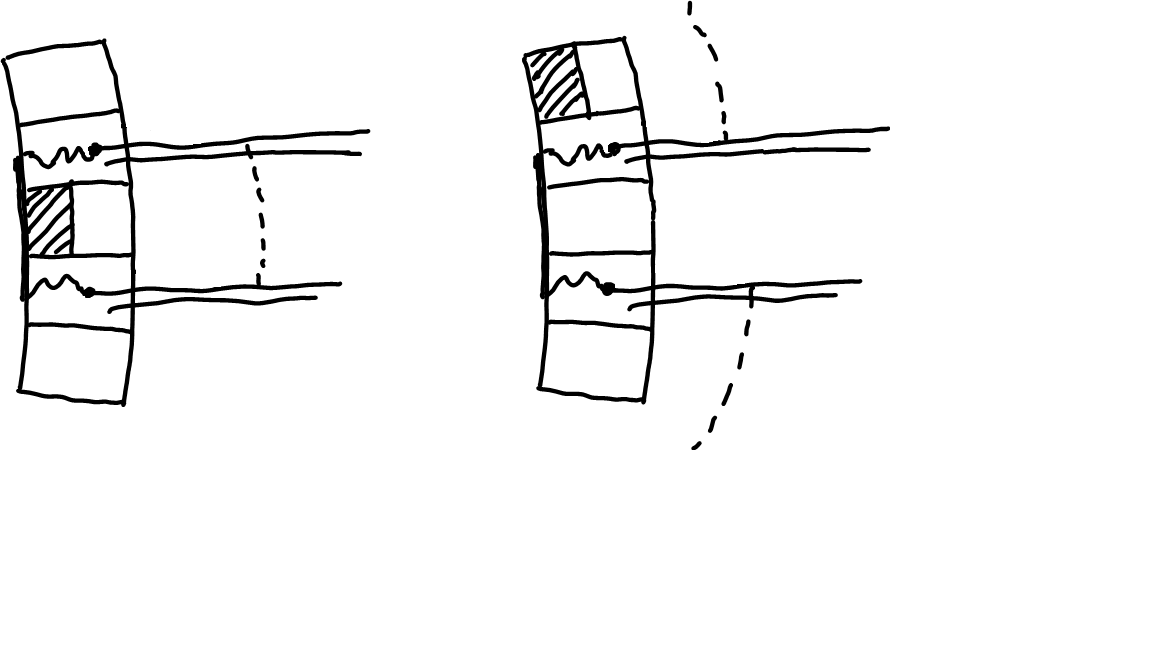}}%
    \put(0.075,0.44){\color[rgb]{0,0,0}\makebox(0,0)[lb]{\smash{$\gamma'$}}}%
    \put(0.08,0.295){\color[rgb]{0,0,0}\makebox(0,0)[lb]{\smash{$\gamma$}}}%
    \put(0.125,0.443){\color[rgb]{0,0,0}\makebox(0,0)[lb]{\smash{$z'$}}}%
    \put(0.12,0.289){\color[rgb]{0,0,0}\makebox(0,0)[lb]{\smash{$z$}}}%
    \put(0.04,0.36){\color[rgb]{0,0,0}\makebox(0,0)[lb]{\smash{$\Gamma$}}}%
    \put(0.495,0.36){\color[rgb]{0,0,0}\makebox(0,0)[lb]{\smash{$\Gamma$}}}%
    \put(0.53,0.44){\color[rgb]{0,0,0}\makebox(0,0)[lb]{\smash{$\gamma'$}}}%
    \put(0.535,0.295){\color[rgb]{0,0,0}\makebox(0,0)[lb]{\smash{$\gamma$}}}%
    \put(0.575,0.443){\color[rgb]{0,0,0}\makebox(0,0)[lb]{\smash{$z'$}}}%
    \put(0.57,0.29){\color[rgb]{0,0,0}\makebox(0,0)[lb]{\smash{$z$}}}%
    \put(0.118,0.37){\color[rgb]{0,0,0}\makebox(0,0)[lb]{\smash{$Q_{0,0}$}}}%
    \put(0.559,0.485){\color[rgb]{0,0,0}\makebox(0,0)[lb]{\smash{$Q_{0,2}$}}}%
    \put(0.2,0.37){\color[rgb]{0,0,0}\makebox(0,0)[lb]{\smash{$\Omega$}}}%
    \put(0.47,0.45){\color[rgb]{0,0,0}\makebox(0,0)[lb]{\smash{$\Omega$}}}%
    \put(0.285,0.37){\color[rgb]{0,0,0}\makebox(0,0)[lb]{\smash{$\Delta\subset A_R(f)$}}}%
    \put(0.68,0.5){\color[rgb]{0,0,0}\makebox(0,0)[lb]{\smash{$\Delta\subset A_R(f)$}}}%
    \put(0.10,0.185){\color[rgb]{0,0,0}\makebox(0,0)[lb]{\smash{$Q_0$}}}%
    \put(0.55,0.185){\color[rgb]{0,0,0}\makebox(0,0)[lb]{\smash{$Q_0$}}}%
\end{picture}%
\endgroup%
\end{center}
\vspace*{-2.5cm}
\caption{\small Two possible locations of the component $\Delta$ and the domain $\Omega$}
\label{pic3}
\end{figure}

In either case, we deduce that $f(\Omega)$ is a bounded domain that contains an open annulus of the form $A(\tfrac12r_1,\tfrac12 r_1+\eps)$, for some $\eps>0$.

Now,
\[
\partial f(\Omega)\subset f(\partial \Omega)\subset f(\Gamma)\cup f(\Delta),
\]
and we claim that $f(\Gamma)$ does not meet the outer boundary component, $C$ say, of $\partial f(\Omega)$. This is clearly the case for the part of $\Gamma$ lying in the inner edge of $Q_0$, which maps to $\{z:|z|=\tfrac12 r_1\}$, and is the case for $f(\gamma)$ and $f(\gamma')$ because each of these paths has a preimage path lying entirely in $Q_{0,0}$ or $Q_{0,2}$, and hence in~$\Omega$.

Therefore,
\[
C\subset f(\Delta)\subset A_{M(R)}(f),
\]
so $A_{M(R)}(f)$ has a bounded complementary component and hence is a {\sw}, by \cite[Theorem~1.4]{RS09}. This is a contradiction. Thus all the components of $A_R(f)$ that correspond to the uncountably many Eremenko points with distinct Wiman--Valiron itineraries are disjoint.

We have now shown that if $A_R(f)$ is {\it never} a spider's web, then for each $r_0\ge R_1(f)$ the interval $(r_0,2r_0)$ contains a value of~$R$ such that $A_R(f)$ has uncountably many components. To complete the proof of Theorem~\ref{AR(f)}, we deduce that there exists a dense set of values~$R$ in $(R(f),\infty)$ with this property. To do this we use the fact (see \cite[Lemma~2.2]{RS09b}) that there exists $R_2(f)>0$ such that
\begin{equation}\label{Mrc}
M(r^c)\ge M(r)^c \qfor r\ge R_2(f),\; c>1.
\end{equation}
Suppose that $(R',R'')$ is any non-empty open subinterval of $(R(f),\infty)$. Then, it follows from \eqref{Mrc} that there exists $N\in\N$ such that $M^N(R'')\ge 2M^N(R')\ge 2R_1(f)$, so the interval $(M^N(R'),M^N(R''))$ contains an interval of the form $(r_0,2r_0)$, where $r_0\ge R_1(f)$.  By the earlier part of the proof, we can find $R_0\in (r_0,2r_0)$ such that $A_{R_0}(f)$ has uncountably many components. Now put $R=M^{-N}(R_0)$. Then $R\in (R',R'')$ and we claim that $A_{R}(f)$ has uncountably many components. For if $\Gamma_0$ and $\Gamma_0'$ are distinct components of $A_{R_0}(f)$, and $\Gamma$ and $\Gamma'$ are components of $f^{-N}(\Gamma_0)$ and $f^{-N}(\Gamma_0')$, respectively, then $\Gamma$ and $\Gamma'$ are contained in $A_R(f)$, by \eqref{equiv}, and must be contained in distinct components of $A_R(f)$, since~$f$ is continuous. Hence $A_R(f)$ has uncountably many components. This completes the proof of Theorem~\ref{AR(f)}.

\section{Properties of \mcfc s}
\label{mcwds}\setcounter{equation}{0}
In this section we recall some known properties of multiply connected wandering domains, which are needed to prove Theorems~\ref{varied} and~\ref{mc-IJ}. In the first result, we give some basic properties, including the result of Baker that, for transcendental entire functions, multiply connected wandering domains are the only multiply connected Fatou components.
\begin{lemma}\label{basic-mcwds}
Let~$f$ be a {\tef}, let~$U$ be a multiply connected Fatou component of~$f$, let $U_n=f^n(U)$ for $n\in \N_0$, and suppose that $R>R(f)$. Then~$U$ is a bounded wandering domain and, more precisely,
\begin{itemize}
\item[(a)]
each $U_n$, $n\in\N$, is a bounded Fatou component of~$f$;
\item[(b)]
$U_{n+1}$ surrounds $U_n$ for sufficiently large $n$, and dist\,$(\partial U_n,0) \to \infty$ as $n\to\infty$;
\item[(c)]
$\overline{U_n}\subset A_R(f)$, for sufficiently large~$n$, and $A_R(f)$ is a {\sw};
\item[(d)]
all the components of $J(f)$ and hence of $A_R(f)\cap J(f)$ are bounded;
\item[(e)]
each component of $\partial U_n$, for sufficiently large~$n$, is contained in a distinct component of $A_R(f)\cap J(f)$;
\item[(f)]
$f$ has no exceptional points, that is, no points with a finite backward orbit.
\end{itemize}
\end{lemma}
See \cite[Theorem~3.1]{iB84} for properties~(a) and~(b), and \cite[Theorem~4.4, Corollary~6.1 and Theorem~1.3]{RS09} for properties~(c),~(d) and~(e). Property~(f) holds because if~$f$ has an exceptional point $\alpha$, then~$f$ has the form $f(z)=\alpha+(z-\alpha)^m \exp(g(z))$, where $m\ge 0$ and~$g$ is entire \cite[Section~2.2]{wB93}, and it follows that~$f$ has asymptotic value~$\alpha$, which is impossible by property~(b).

In order to define the notion of inner connectivity, which is used in the statement of Theorem~\ref{varied}, we need the following result \cite[Theorem~1.3]{BRS11}. This strengthens an earlier result of Zheng \cite{jhZ02} showing that \mcfc s contain large annuli.
\begin{lemma}\label{large-annuli}
Let $f$ be a {\tef} with a {\mcfc}~$U$, let $z_0\in U$ and put $r_n=|f^n(z_0)|$ and $U_n=f^n(U)$ for $n\in\N_0$. Then there exist $\alpha>0$ and sequences $(a_n)$ and $(b_n)$ with
\[
0<a_n<1-\alpha<1+\alpha<b_n, \qfor n\in\N_0,
\]
such that, for sufficiently large $n \in \N$,
\[
B_n=A(r_n^{a_n}, r_n^{b_n})\subset U_n.
\]
Moreover, for every compact subset $C$ of $U$, we have $f^n(C)\subset B_n$ for $n\ge N(C)$.
\end{lemma}
In view of this last property we often describe the large annuli $B_n$ as `absorbing'.

We then define the {\it inner connectivity} of $U_n$ to be the connectivity of the domain $U_n\cap\{z: |z|<r_n\}$ and the {\it outer connectivity} of $U_n$ to be the connectivity of the domain $U_n\cap\{z: |z|>r_n\}$. We also define the {\it outer boundary component} of a bounded domain~$U$ to be the boundary of the unbounded component of $\C \setminus U$, denoted by $\partial_{\rm out}U$, and the {\it inner boundary component} of~$U$ to be the boundary of the component of $\C \setminus U$ that contains~0, if there is one, denoted by $\partial_{\rm inn}U$.

The inner connectivity of a \mcfc\ can behave in one of two ways, given by the following lemma; see \cite[Theorem~8.1(b)]{BRS11}.
\begin{lemma}\label{event-conn}
Let~$f$ be a \tef, let~$U$ be a \mcfc\ of~$f,$ and let $U_n=f^n(U)$ for $n\in \N_0$. Then there exists $N\in\N$ such that exactly one of the following holds:
\begin{itemize}
\item[(a)]
$U_n$ has infinite inner connectivity for all $n\ge N$;
\item[(b)]
$U_n$ has finite inner connectivity for all $n\ge N,$ which decreases with~$n$, eventually reaching the value~2.
\end{itemize}
\end{lemma}
{\it Remark}\;\;It is clear from this lemma that the concept of {\it eventual inner connectivity} of a \mcfc\ (see \cite{BRS11}) is well defined, and that the eventual inner connectivity can take the values infinity or~2.

To prove Theorems~\ref{varied} and~\ref{mc-IJ}, we use another property of {\mcfc}s, proved as part of \cite[Lemma~3.2]{RS15}.
\begin{lemma}\label{nested}
Let~$f$ be a \tef, let $U$ be a \mcfc\ of $f,$ and let $U_n=f^n(U)$ for $n\in \N_0$. Then there exists $N\in\N$ and a sequence of annuli $B'_n=A(r'_n,r''_n) \subset U_n$, for $n\ge N$, such that
\begin{equation}\label{B'n}
f(B'_n)\subset B'_{n+1}, \qfor n\ge N.
\end{equation}
In particular, if $R\in (r'_N,r''_N)$, then
\begin{equation}\label{Rn}
\{z:|z|=M^{n-N}(R)\}\subset U_n, \qfor n\ge N.
\end{equation}
\end{lemma}

The final result in this section describes the three possible types of complementary components that a \mcfc\ can have. We discuss the possible existence of these types of components in the next section.
\begin{lemma}\label{types}
Let~$f$ be a \tef\ with a \mcfc\ $U$ and let $K$ be a bounded complementary component of~$U$.
\begin{itemize}
\item[(a)] For all $n\in\N$, the set $f^n(K)$ is a bounded complementary component of $f^n(U)$.
\item[(b)]The component~$K$ is of one of the following types:
\begin{itemize}
\item[1.]
the interior of~$K$ meets $J(f)$ and, for sufficiently large~$n\in\N$, the set $f^n(K)$ is the complementary component of $f^n(U)$ that contains~0;
\item[2.]
the interior of~$K$ is a union of Fatou components;
\item[3.]
$K$ has empty interior.
\end{itemize}
\item[(c)] If~$K$ is of type~2 or type~3, then every point of~$\partial K$ is the limit of points lying in distinct type~1 complementary components.
\end{itemize}
\end{lemma}
\begin{proof}
The result of part~(a) may be known, but we include a proof for completeness. It is sufficient to prove this for the case $n=1$.

Since $f(K)$ is connected there is a unique complementary component, $L$ say, of $V=f(U)$ such that $f(K)\subset L$. We show that $f(K)=L$.

Let $V_n, n\in\N$, be a smooth exhaustion of~$V$; that is, $V_n$ are smooth domains such that $\overline{V_n}\subset V_{n+1}$, for $n\in\N$, and $\bigcup_{n\in\N} V_n=V$. Then~$L$ lies in a unique component, $H_n$ say, of $\C\setminus \overline{V_n}$, for each $n\in\N.$ Each $H_n$ is a Jordan domain with its boundary in~$V$, $\overline{H_{n+1}}\subset H_n$, for $n\in \N$, and $\bigcap_{n\in\N} \overline {H_n}=L$, since $(V_n)$ is an exhaustion of~$V$.

Now let $G_n$, $n=1,2,\ldots,$ denote the component of $f^{-1}(H_n)$ that contains~$K$. Then $\partial G_n \subset U$ and $f:G_n \to H_n$ is a proper map, so $\overline{G_{n+1}}\subset G_n$, for $n\in\N$. Then $K'=\bigcap_{n\in\N} \overline {G_n}$ is a compact connected set such that $f\left(K'\right)=L$. Indeed, $f(\ol{G_n})=\ol{H_n}$, for $n\in\N$, so the inclusion
\[
f\left(\bigcap_{n\in\N}\ol{G_n}\right)\subset \bigcap_{n\in\N}\ol{H_n}
\]
is clear, and the reverse inclusion also holds since if $w\in \bigcap_{n\in\N}\ol{H_n}$, then there exists $z_n\in \ol{G_n}$ such that $f(z_n)=w$, for all $n\in\N$, and hence $f(z)=w$ for some $z\in \bigcap_{n\in\N}\ol{G_n}$, by compactness.

We now show that $K'=K$. Clearly, $K\subset K'$ and also $K' \subset \C\setminus U$, since $f(K')=L \subset \C\setminus V$, so $K'=K$ and hence $f(K)=L$, as required.

To prove part~(b), suppose that the interior of~$K$ meets $J(f)$. Since $f$ has no exceptional values, the backward orbit of~0 accumulates at every point of $J(f)$, so we deduce that int\,$K$ must contain a point~$z$ such that for some $n\in\N$ we have $f^n(z)=0$. It follows by part~(a) that $f^n(K)$ is the complementary component of $f^n(U)$ that contains~0.

Part~(c) also follows immediately from the fact that the backward orbit of~0 accumulates at every point of $J(f)$.
\end{proof}

\section{complementary components of multiply connected\\ wandering domains}
\label{mcwds-bdrycmpnts}\setcounter{equation}{0}
In this section we prove the following result, which arose from discussions with Markus Baumgartner and Walter Bergweiler. We use Theorem~\ref{inner} in the proof of Theorem~\ref{varied}, and it also has considerable interest in its own right. For example, it shows that in some cases a \mcfc\ can have uncountably many complementary components of type~3, that is, ones with no interior. It was not previously known whether such complementary components could exist.

\begin{theorem}\label{inner}
Let~$f$ be a \tef\ with a \mcfc\ $U$, let~$N$ be so large that the inner and outer connectivity of $U_N$ are defined, and let $B_N$ be defined as in Lemma~\ref{large-annuli}.
\begin{itemize}
\item[(a)]
If $U_N$ has infinite inner connectivity, then
\begin{itemize}
\item[(i)]
$U_N$ has uncountably many complementary components that accumulate at the inner boundary component of~$U_N$;
\item[(ii)]
$U_N$ has uncountably many complementary components with no interior (type~3), as has~$U$;
\item[(iii)]
the outer connectivity of $U_N$ is either~2 or uncountable.
\end{itemize}
\item[(b)]
If $U_N$ has finite inner connectivity, then the outer connectivity of~$U_N$ is finite or countable, and the complementary components accumulate nowhere in~$\overline{U_N}$ except possibly at the outer boundary component of~$U_N$.
\end{itemize}
\end{theorem}
As far as we know, it is an open question whether a {\mcfc} can have type~2 complementary components and also whether it can have complementary components that are {\it singleton} sets.

We have the following corollary of Theorem~\ref{inner} and Lemma~\ref{types}\,(c), which was also given in Baumgartner's PhD thesis \cite[Theorem~3.1.25]{mB15}; see the remark after Lemma~\ref{event-conn} for the meaning of `eventual inner connectivity'.

\begin{corollary}\label{markus}
Let~$f$ be a \tef\ with a \mcfc\ $U$. Then~$U$ has eventual inner connectivity~2 if and only if all the complementary components of~$U$ are of type~1.
\end{corollary}

Examples of {\tef}s with {\mcfc}s having either infinite inner connectivity or finite inner connectivity were given in \cite[Remarks following Theorem 1.3]{BZ11}. Other examples with finite inner connectivity were given in \cite{KS08} and in \cite{cB}.

In the proof of Theorem~\ref{inner} we use the following simple topological lemma; see \cite[Lemma~1]{RS09a}.
\begin{lemma}\label{toplemma}
Let $E_n$, $n\ge 0$, be a sequence of compact sets in $\C$ and $f:\C\to\hat{\C}$ be a continuous function such that
\begin{equation}\label{contains}
f(E_n)\supset E_{n+1},\quad\text{for } n\ge 0.
\end{equation}
Then there exists~$\zeta$ such that $f^n(\zeta)\in E_n$, for $n\ge 0$.

If~$f$ is also meromorphic and $E_n\cap J(f)\ne\emptyset$, for $n\ge 0$, then there exists $\zeta\in J(f)$ such that $f^n(\zeta)\in E_n$, for $n\ge 0$.
\end{lemma}
\begin{proof}[Proof of Theorem~\ref{inner}] (a)\; Put $n_0=N$. Since $U_{n_0}$ has infinite inner connectivity and $J(f)$ is closed, we deduce from Lemma~\ref{Newman}\,(b) that there exists a Jordan curve $\gamma_{0}$ in~$U_{n_0}$ which surrounds at least one component of $U_{n_0}^c$ and does not surround~0. Then, by Lemma~\ref{large-annuli}, together with the argument principle, there exists $n_1\in \N$ such that $f^{n_1-n_0}(\gamma_{0})$, lies in the absorbing annulus $B_{n_1}$ and winds at least once round~0. By Lemma~\ref{event-conn}, we can take disjoint Jordan curves $\gamma_{n_1,1}$ and $\gamma_{n_1,2}$ in~$U_{n_1}$, with disjoint interiors, which each surround at least one component of $U_{n_1}^c$, do not surround 0, and lie in the bounded component of $B^c_{n_1}$, and then repeat the process above for each of $\gamma_{n_1,1}$ and $\gamma_{n_1,2}$.

Continuing in this way, we can construct a strictly increasing sequence of positive integers $(n_m)$ and, for $m\in\N$, disjoint Jordan curves $\gamma_{n_m,1}$ and $\gamma_{n_m,2}$, with disjoint interiors, which each surround at least one component of $U_{n_m}^c$, do not surround~0, and lie in the bounded component of $B^c_{n_m}$, such that
\[
f^{n_{m+1}-n_{m}}(\gamma_{n_m,j})\;\text{ surrounds }\; \gamma_{n_{m+1},k},\qfor m\in\N,\;j,k=1,2.
\]

We now show that there exist points in $J(f)$ whose images under $f^{n_m}$ lie in the interior of any specified choice of the Jordan curves $\gamma_{n_m,j_m}$, for $m\in \N$, $j_m\in \{1,2\}$. To prove this we consider the sequence of compact sets $(E_n)$, defined as follows:
\begin{itemize}
\item
for $n=n_m, m\ge 0$, we take $E_n$ to be the union of $\gamma_{n_m,j_m}$ and int\,$\gamma_{n_m,j_m}$;
\item
for $n_m<n<n_{m+1}$,  $m\ge 0$, we take $E_n$ to be $f^{n-n_m}(E_{n_m})$.
\end{itemize}
It is clear that the sequence $(E_{n_0+k})$, $k\ge 0$, satisfies all the hypotheses of Lemma~\ref{toplemma}, so there exists~$\zeta \in J(f)$ such that $f^{k}(\zeta)\in E_{n_0+k}$ for $k\ge 0$ and, in particular,
\begin{equation}\label{zetaorbit}
f^{n_m-n_0}(\zeta) \;\text{ is surrounded by }\; \gamma_{n_m,j_m},\qfor m\in \N.
\end{equation}
Since we have two choices of Jordan curve at each stage (after the first), this gives rise to uncountably many points of $J(f)$, each of which has the property that its images lie in the interiors of the specified Jordan curves. Each such point, $\zeta$ say, must be contained in a complementary component, $K_{\zeta}$ say, of~$U_{n_0}$, and we claim that if \eqref{zetaorbit} holds, then
\begin{equation}\label{Eorbit}
f^{n_m-n_0}(K_{\zeta}) \;\text{ is surrounded by }\; \gamma_{n_m,j_m},\qfor m\in\N,
\end{equation}
from which it follows that all such complementaryc components $K_{\zeta}$, arising from points with different `itineraries' $(j_m)$, are distinct. Hence $U_{n_0}=U_N$ has uncountably many complementary components, as required.

To deduce \eqref{Eorbit} from \eqref{zetaorbit}, we note that, for $m\in\N$, each complementary component of~$U_{n_0}$ must map under $f^{n_m-n_0}$ onto a complementary component of $U_{n_m}$, by Lemma~\ref{types}. It follows, in particular, that the complementary component $K_{\zeta}$ must map under $f^{n_m-n_0}$ to the complementary component of $U_{n_m}$ that contains $f^{n_m}(\zeta)$, so this complementary component of $U_{n_m}$ must be surrounded by $\gamma_{n_m,j_m}$. This proves \eqref{Eorbit}.

These complementary components of $U_{n_0}$ must accumulate at the inner boundary component of~$U_{n_0}$ for otherwise we could find a Jordan curve $\gamma \subset U_{n_0}$ that surrounds the inner boundary component of $U_{n_0}$ but no other boundary components. For~$n$ sufficiently large $f^n(\gamma)$ must lie in $B_{n_0+n}$ and wind at least once round~0, which contradicts the fact that $U_{n_0+n}$ has infinite inner connectivity. This proves part~(a)(i) and part~(a)(ii) follows immediately. To prove part~(a)(iii), we observe that if a complementary component, $K$ say, of $U_{n_0}$ exists outside $B_{n_0}$ and $\gamma$ is a Jordan curve in $U_{n_0}$ that surrounds~$K$, then for~$n$ sufficiently large $f^n(\gamma)$ must lie in $B_{n_0+n}$ and wind at least once round~0, and hence $\gamma$ must surround uncountably many complementary components of $U_{n_0}$.

To prove part~(b) we suppose that the inner connectivity of $U_N$ is finite. Let~$\gamma$ be a Jordan curve in~$U_N$ that surrounds at least one boundary component of $U_N$. Then there exists $n\in\N$ such that $f^n(\gamma)$ lies in $B_{N+n}$ and winds at least once round~0. Then $f^n(\gamma)$ winds round at most finitely many components of $U^c_{N+n}$, so~$\gamma$ surrounds at most finitely many components of $U_N^c$. It follows that~$U_N$ has at most countably many complementary components, so the outer connectivity of $U_N$ is at most countable, and also that the complementary components of $U_N$ do not accumulate at any point of $\overline{U_N}$ except possibly at the outer boundary component of $U_N$.
\end{proof}

\section{Proofs of Theorems~\ref{varied} and~\ref{mc-IJ}}
\setcounter{equation}{0}\label{varied-mc-IJ}
In this section we prove our results about the possible numbers of components of $A_R(f)\cap J(f)$, $A(f)\cap J(f)$ and $I(f)\cap J(f)$ in the case when~$f$ has a {\mcfc}. We begin by proving Theorem~\ref{varied}.
\begin{proof}[Proof of Theorem~\ref{varied}]
Part~(a) states that if~$U$ is a \mcfc\ of a \tef\ $f$ with infinite inner connectivity and $R>R(f)$, then $A_R(f)\cap J(f)$ and $A(f)\cap J(f)$ have uncountably many components. By Lemma~\ref{basic-mcwds}~(c) we can take~$N$ so large that $\overline{U_N}=\overline{f^N(U)}\subset A_R(f)$. By Theorem~\ref{inner}~(a)(i), we know that $U_N$ has uncountably many complementary components. The boundaries of these complementary components are subsets of $A_R(f)\cap J(f)$ and $A(f)\cap J(f)$, so the result follows.

Part~(b) states that there exists a {\tef}~$f$ with a {\mcfc} and $R>R(f)$ such that $A_R(f)\cap J(f)$ and $A(f)\cap J(f)$ each have only countably many components. We show that this is the case for a remarkable example constructed by Bishop \cite{cB} of a \tef\ whose Julia set has dimension~1. In this example, there is a \mcfc\ $U$ whose forward orbit $U_n=f^n(U)$ has the following topological properties. For $n\ge 0$,
\begin{itemize}
\item the boundary components of $U_n$ are all Jordan curves;
\item the inner boundary component of $U_{n+1}$ is identical to the outer boundary component of $U_n$;
\item the outer connectivity of $U_n$ is countably infinite and the inner connectivity is~2.
\end{itemize}
By Lemma~\ref{Rn}, we can also assume that there exists $R>R(f)$ such that
\begin{equation}\label{Bishop-Rn}
\{z:|z|=M^n(R)\} \subset U_n,\qfor n\ge 0.
\end{equation}
In fact Bishop's proof in \cite{cB} gives the property \eqref{Bishop-Rn} as a part of the construction.

We claim that any point $\zeta\in A(f)\cap J(f)$ must lie in one of the countably many boundary components of one of the domains $U_n$ or the pre-image of such a boundary component. If not, then it follows, by the properties of the wandering domain given above, that, for all sufficiently large $n\in\N$, the point $f^n(\zeta)$ must lie in  the interior of a type~1 complementary component of $U_{k(n)}$. Thus, by Lemma~\ref{types}\,(b)(i), there exists a sequence $(n_j)$ such that
\[
k(n_j)\le k(n_{j-1}) +n_j-n_{j-1}-1,\qfor j \in \N,
\]
and hence
\[
k(n_j)-n_j \le d - j,\qfor j \in\N,
\]
where $d=k(n_0)-n_0$. Thus
\[
f^{n_j}(\zeta)\in \widetilde{U}_{d+n_j-j},\qfor j\in \N.
\]
Here the notation~$\widetilde{V}$ denotes the union of~$V$ with its bounded complementary components. Therefore, by \eqref{Bishop-Rn},
\[
|f^{n_j}(\zeta)|\le M^{d+n_j-j+1}(R),\qfor j\in\N,
\]
which implies that $\zeta\notin A(f)$, a contradiction. Hence there are only countably many components of $A(f)\cap J(f)$, and similarly only countably many components of $A_R(f)\cap J(f)$. This completes the proof of Theorem~\ref{varied}.
\end{proof}
We remark that a large class of \tef s with topological properties similar to Bishop's example was constructed by Baumgartner~\cite{mB15}.

Finally we prove Theorem~\ref{mc-IJ}. Here we use a variation of an argument we introduced in \cite{RS09a} and~\cite{RS15}.
\begin{proof}[Proof of Theorem~\ref{mc-IJ}]
Theorem~\ref{mc-IJ} states that for any {\tef}~$f$ with a {\mcfc}, the set $I(f)\cap J(f)$ has uncountably many components. Given such a function~$f$, let $B'_n$, $n\ge 0$, be the open annuli given by Lemma~\ref{nested} and for each $n\ge 0$ let $E_n$ denote the closed annulus lying precisely between $B'_n$ and $B'_{n+1}$. Then, by \eqref{B'n},
\[
\partial f(E_n)\subset f(\partial E_n)\subset \overline{B'_{n+1}}\cup \overline{B'_{n+2}}, \;\;\text{for } n\ge 0,
\]
so
\begin{equation}\label{all}
f(E_n)\supset E_{n+1},\qfor  n\ge 0.
\end{equation}
Also, since the annuli $B'_n$ lie in distinct Fatou components of~$f$, we deduce that
\begin{equation}\label{EintJ}
E_n\cap J(f) \ne \emptyset,\qfor n\ge 0.
\end{equation}

Next we put
\[
E'_n=B'_n\cup E_n\cup B'_{n+1},\qfor n\ge 0,
\]
and let $F_n$ denote the bounded component of $\C\setminus E'_n$. Then it follows from \eqref{B'n} that, for each $n\ge0$, we have exactly one of the following possibilities:
\begin{equation}\label{onto}
f(E'_n)\subset E'_{n+1},
\end{equation}
or
\begin{equation}\label{notonto}
f(E'_n)\supset F_{n+1},\;\;\text{so }\; f(E_n)\supset F_{n+1}\supset E_n.
\end{equation}

If~(\ref{onto}) holds for all $n\ge N$, say, then each $E'_n$, $n\ge N$, is contained in the Fatou set of~$f$, by Montel's theorem, and this contradicts the fact that each $E_n$ and hence each $E'_n$ meets $J(f)$. Thus there is a strictly increasing sequence $n_j\ge 0$ such that \eqref{notonto} holds for $n=n_j$, $j\in\N$, so
\begin{equation}\label{some}
f(E_{n_j})\supset E_{n_j},\qfor j\in \N.
\end{equation}
We now observe that there are uncountably many increasing sequences~$s$ of non-negative integers, each of which includes all the non-negative integers and some repetitions of the integers $n_j$, $j\in \N$. For each of these sequences, the properties \eqref{all}, \eqref{EintJ} and \eqref{some} allow us to apply  Lemma~\ref{toplemma} to give a point in $J(f)$ whose orbit passes through the annuli $E_n$ in a manner determined by the sequence~$s$. For two such distinct sequences~$s$, we obtain (since the annuli $B'_n$ all lie in $F(f)$) two distinct components of $I(f)\cap J(f)$. Since there are uncountably many such distinct sequences, there are uncountably many distinct components of $I(f)\cap J(f)$, as required.
\end{proof}

\section{Open questions}
\setcounter{equation}{0}\label{outstanding}
In this final section we discuss several interesting questions, which arise in connection with our new results.

The main question that we sought to address in this paper was the following.

\begin{question}\label{qn1}
Let $f$ be a {\tef}. For each of the sets $I(f)$, $A(f)$ and $A_R(f)$, where $R>R(f)$, is it the case that the set is either connected or it has uncountably many components?
\end{question}

Theorem~\ref{I(f)} gives a partial answer to this question for $I(f)$, since it states that if $I(f)^c$ contains an unbounded continuum, in particular if $I(f)$ is disconnected, then the set $I(f)\setminus D$, where~$D$ is any open disc that meets $J(f)$, has uncountably many unbounded components, and Theorem~\ref{A(f)} gives a similar partial result for $A(f)$. These results raise the following question about $I(f)$, and there is a similar question about $A(f)$.

\begin{question}\label{qn2}
Does there exist a {\tef} $f$ such that $I(f)$ is connected and $I(f)^c$ contains an unbounded continuum?
\end{question}

The function $f(z)=e^z$ has the property that $I(f)$ is connected and there is an unbounded connected set in the complement of $I(f)$; see \cite[Example~2]{OS}. Note that in our proof of Theorem~\ref{I(f)} we make strong use of the fact that the unbounded connected set~$\Gamma$ in $I(f)^c$ is closed.

Theorem~\ref{AR(f)} gives a complete answer to Question~\ref{qn1} for $A_R(f)$ for many values of~$R$, and furthermore shows that for these values of~$R$ if $A_R(f)$ is connected, then it is a spider's web. In this situation, we can easily deduce that if $A_R(f)$ is disconnected, then it has uncountably many components for {\it all} $R>R(f)$ whenever the answer to the following intriguing question is `yes'.

\begin{question}\label{qn3}
Let $f$ be a {\tef} and let $A_0$ be a component of $A_R(f)$, for some $R>R(f)$. Must $A_0$ meet $A_{R'}(f)$ for every $R'>R$?
\end{question}

If the answer to Question~\ref{qn3} is `yes', then with somewhat more work we can also show that if $I(f)$ is disconnected, then $I(f)$ has uncountably many components (not necessarily all unbounded).

We now state two questions about {\mcfc}s.

\begin{question}\label{qn4}
Let $f$ be a {\tef} and let~$U$ be a {\mcfc} of~$f$. Is it possible for~$U$ to have complementary components of type~2, that is, ones with interior that is a union of Fatou components?
\end{question}

Corollary~\ref{markus} shows that if such type~2 components exist, then~$U$ must have infinite inner connectivity.

Finally, in Theorem~\ref{inner}\,(a)\,(ii) we showed that if a {\tef}~$f$ has a {\mcfc}~$U$ with infinite inner connectivity, then~$U$ has uncountably many type~3 complementary components, that is, ones with no interior.

\begin{question}\label{qn5}
Let $f$ be a {\tef} and let~$U$ be a {\mcfc} of~$f$ with infinite inner connectivity. Is it possible, or indeed necessary, that~$U$ has (uncountably many) singleton complementary components?
\end{question}

{\it Acknowledgement}\quad
The authors are grateful to Markus Baumgartner and Walter Bergweiler for discussions which led to Theorem~\ref{varied}\,(a) and the closely related Theorem~\ref{inner}.

\end{document}